\documentclass[english]{extarticle}
\usepackage{booktabs}
\usepackage[T1]{fontenc}
\usepackage[utf8]{luainputenc}
\usepackage{geometry}
\usepackage{color}
\usepackage{float}
\usepackage{amsmath}
\usepackage{amsthm}
\usepackage{amssymb}
\usepackage{stmaryrd}
\usepackage{esint} 
\PassOptionsToPackage{normalem}{ulem}
\usepackage{ulem}
\usepackage{algorithm2e}

\usepackage{graphicx, caption, subcaption}
\floatstyle{ruled}

\newtheorem{thm}{\protect\theoremname}

\newtheorem{rem}[thm]{\protect\remarkname}

\newtheorem{lem}[thm]{\protect\lemmaname}

\newfloat{algorithm}{tbp}{loa}
\providecommand{\algorithmname}{Algorithm}
\floatname{algorithm}{\protect\algorithmname}

\providecommand{\definitionname}{Definition}
\providecommand{\lemmaname}{Lemma}
\providecommand{\propositionname}{Proposition}
\providecommand{\remarkname}{Remark}
\providecommand{\theoremname}{Theorem}

\usepackage{xcolor}

\newcommand{\A}{\mathcal{A}}
\newcommand{\G}{\mathcal{G}}

\newcommand{\E}{\mathcal{E}}
\newcommand{\N}{\mathcal{N}}

\newcommand{\PP}{\mathbb{P}}

\newcommand{\EX}{\mathbb{E}}

\newcommand{\PPP}{\textbf{P}}

\newcommand{\Mu}{\hat{\mu}}
\newcommand{\m}{\tilde{m}}

\newcommand*{\horzbar}{\rule[.5ex]{2.5ex}{0.5pt}}

\usepackage{footnote}
\newcommand{\Pm}{\mathbb{P}}
\makesavenoteenv{tabular}
\makesavenoteenv{table}

\newcommand{\As}{\bar{A}}
\newcommand{\Ap}{\bar{P}}

\usepackage{array}
\newcolumntype{?}{!{\vrule width 1pt}}
\newcolumntype{P}[1]{>{\centering\arraybackslash}p{#1}}

\DeclareMathOperator*{\argmax}{arg\,max}
\DeclareMathOperator*{\argmin}{arg\,min}

\usepackage{latexsym}

\author{{Suhail M. Shah}
\date{}
\thanks{The author is with the department of Electrical Engineering, Indian Institute of Technology-Bombay, India-400076. (e-mail: suhailshah2005@lgmail.com).}}
\begin{document}

\title{Making Simulated Annealing Sample Efficient for Discrete Stochastic Optimization}

\maketitle

\begin{abstract}%
We study the regret of simulated annealing (SA) based approaches to solving discrete stochastic optimization problems. The main theoretical conclusion is that the regret of the simulated annealing algorithm, with either noisy or noiseless observations, depends primarily upon the rate of the convergence of the associated Gibbs measure to the optimal states. In contrast to previous works, we show that SA does not need an increased estimation effort (number of \textit{pulls/samples} of the selected \textit{arm/solution} per round for a finite horizon $T_{\text{hor}}$) with noisy observations to converge in probability. By simple modifications, we can make the \textit{total} number of samples per iteration required for convergence (in probability) to scale as $\mathcal{O}\big(T_{\text{hor}})$. Additionally, we show that a simulated annealing inspired heuristic can solve the problem of stochastic multi-armed bandits (MAB), by which we mean that it suffers a $\mathcal{O}(\log \,T_{\text{hor}})$ regret. Thus, our contention is that SA should be considered as a viable candidate for inclusion into the family of efficient exploration heuristics for bandit and discrete stochastic optimization problems.
\end{abstract}

\section{Introduction}
Consider the following optimization problem:
\begin{equation}\label{eq}
\min_{a \in \A} \mu_a := \EX_{ \omega} \,f(a,\omega),
\end{equation}
where $\A$ is a finite set (which may have some inherent topological structure), $\omega$ is a random variable and $f : \A \to \mathbb{R}$ is a real valued bounded function. The aim in stochastic optimization is to study and solve (\ref{eq}) in an efficient manner.  Finding suitable solution to this problem can be quite hard, since one may lack a direct access to the underlying probability distribution to compute $\mu_a$. In such situations, one may be faced with the task of optimizing only through samples $f(a,\omega)$. Given this, the primary concern in stochastic optimization is to obtain a good solution, i.e.,  $\mu_{a_T}$ is as small as possible for any algorithmic output $a_T$.

The above problem can also be studied in a `bandit setting' (\cite{lai}, \cite{aue}, \cite{bubeck}, \cite{tor}) where the emphasis is placed on the total loss or `regret' incurred by the algorithm. A stochastic bandit is a collection of distributions $\nu:= (P_a \,:a\in\A)$, where $\A$ is the set of available actions with mean-payoff $\mu_a := \int x \,dP_a(x)$. During any time instance $t$, referred to as a round, the learner chooses any action $a_t \in \A$, interacts with the environment following which a loss $X_t$ (or for a maximization problem, a reward) is revealed to the learner sampled from $P_{a_t}$. For instance, the learner may select an arm $a_1 \in \A$ in the first round and incur the loss $X_1$ drawn at random according to $P_{a_1}$. In any subsequent round $n$, the learner can select an arm $a_n$ based on all prior information available till the current round, i.e. the variables $(a_1,X_1,\cdots,a_{n-1},X_{n-1})$ and incur a loss $X_n$ independently of $(a_1,X_1,\cdots,a_{n-1},X_{n-1})$. Any strategy (formally a sequence of probability kernels defined on an appropraite probability space, see definition 4.7 in \cite{tor}) which instructs on how to select an arm at round $n$ based on previous information is referred to as a policy. In bandit optimization, the goal of the learner is to construct policies that minimize the expected cumulative regret, which is the
difference between the expected loss incurred and the loss incurred (in
expectation) by choosing the optimal action. More formally, let
$
a^* := \argmin_{a\in \A} \mu_{a} 
$
be the optimal loss for any round. At round $n$, the cumulative (pseudo) regret of a learner will then be
$$R_n = \EX \sum_{t=1}^n X_t - n \mu_{a^*}, $$
where the expectation is over the learner’s policy (see Section 4.4, \cite{tor} for a formal treatment). 

Bandit optimization encapsulates the exploration/utilization trade-off encountered in many practical situations. Studying the regret of the algorithm can give us a good idea of the number of estimates/samples an algorithm requires (sample complexity) to reach a suitable solution. The sample complexity may take precedence over other measures of performance, when, for instance, obtaining such samples may entail a lengthy simulation. This is frequently the case for most practical applications of importance, including reinforcement learning. 

In this paper, we study the regret of policies based on  Simulated Annealing (SA). To the best of our knowledge, there has been no formal study of SA in a bandit setting. Even for the noiseless setting, which is of importance for combinatorial optimization among other things, a regret analysis is missing. Thus, there exists a significant gap of knowledge in understanding how to efficiently utilize function samples when working with noisy estimates in SA. We take a step in this direction and show that SA can indeed search the solution space  quite efficiently.

\subsection{Related Literature and Contributions}
In this subsection, we briefly cover the prior work done on simulated annealing and clarify the motivation and contributions of the present work. SA was originally proposed in \cite{kirk} 
for finding globally optimum  configurations in large NP-complete problems. The foundational work on the theory of simulated annealing was done in the eighties. We briefly mention some of the works here\footnote{For the sake of brevity, this account is incomplete. We refer the reader to \cite{bert} (and the references within) for a more detailed history.}. In the seminal paper of \cite{hajek} the remarkable notion of depth was introduced and the algorithm was shown to converge under the assumptions of reversibility. In particular, \cite{hajek} was the first to pin down the minimum $\gamma$ for which the alogrithm would converge in probability for a cooling schedule of the form $T_t = \frac{\gamma}{\log(t+1)},$ where $t$ is the iteration count. \cite{mitra} provided finite time convergence bounds on the probability of selecting any state in terms of the graph radius. This work, along with \cite{gidas}, was one of the first to use the theory of time inhomogeneous Markov chains to study finite time performance.  \cite{tsitsiklis} used the perturbation theory of Markov chains to study the convergence of SA while \cite{holley} studied SA using Sobolev inequalities. The latter obtained the same characterization as \cite{hajek} of the cooling schedule with remarkably concise proofs by using the Dirichlet machinery. We will adopt their approach in the later part of this work.

The program of analysing SA with noisy observations  began with \cite{gelfand}. It was established that if the noise variance (in the sample observations) decreased by atleast a rate of $o(T^{-1}_t) = o(\log t)$, the asymptotic behaviour would be unaffected (the cooling schedule bounds reamined the same as in \cite{hajek}). In \cite{gutjahr}, a convergence in total variation was established with the requirement that the number of samples per iteration increase to the order $\mathcal{O}(n^{2+\epsilon}),$ $\epsilon>0$. A recent work by \cite{bout}, clarified that $n^\alpha$, $\alpha>0$ samples could also be sufficient to establish convergence in probability. This was also the first work, to the best of our knowledge, that provided upper bounds on the convergence \textit{rate} of SA with noisy observations. There are  other works, too numerous to mention here, that also study simulated annealing in a noisy setting. We refer the reader to Section 3 in \cite{branke} for a more comprehensive account.


\textit{Contributions of the present work :}  In this work, we study and establish bounds on the regret of SA with both noisy and noiseless observations to show that it can be used as an efficient exploration heuristic when one is faced with noisy estimates. For this setting, previous works maintain that the total number of samples required to guarantee convergence in probability for SA keep on increasing (sometimes as much as $\mathcal{O}(t^{2})$). This does not paint a flattering portrait of SA and makes it seem unappealing in comparison to the state of the art. Additionally, such a requirement is impossible to satisdy in most cases in an online scenario which is of primary interest for noisy setting SA. Nevertheless, as we shall see, SA can be a lot more sample efficient (in fact, it has a constant sample complexity), with simple modifications like keeping track of an empirical average of past samples at each arm and injecting diminishing extraneous noise to the algorithm or  allocating a very small fraction of the budget to uniformly search all the arms. The finite time horizon issue can be handled via the doubling trick (see Section 2.3, \cite{lug}). Our contributions are as follows: 
\begin{enumerate}
\item[(i)] In section 3, we establish an upper bound on the regret of SA with exact observations. This, in particular, is relevant for many combinatorial optimization problems where SA is used routinely to obtain (globally) optimal solutions (\cite{john1}, \cite{john2}, \cite{koul}). We show that there are two components in the regret bound: (i) A fast decaying transient component which depends upon the energy landscape of the function being minimized. (ii) A `steady state' component whose decay is governed by how fast the Gibbs measure converges to the state of the optimal states.

\item[(ii)]  In Section 4, we study modifications of SA in the noisy case for general graphical structures (with some assumptions). This is of importance for graph based optimization and bandit optimization with graphical constraints, wherein during any round the choice of the next arm is constrained to a certain subset of the arm set depending upon the current arm. Such a situation arises in many practical applications such as in remote sensing, e-commerce portals and image detection (see Section 1, \cite{avr} for more details). We show by a simple example that the standard bandit algorithms may fail in such a scenario. On the theoretical side, we show that the regret of the algorithm is primarily dependent upon the rate of convergence of the associated Gibbs measure to the optimal states by simple modifications to the core SA algorithm. This substantially improves on the prior work with respect to sample complexity.

\item[(iii)]  In Section 5, we propose an SA inspired heuristic to solve the multi armed bandit problem. We show that the regret can be upper bounded by $\mathcal{O}( \log n)$, where $n$ is the time horizon. This shows that SA can solve the multi armed bandit problem with optimal rates within constant factors. We note that this setting can be considered as a fully connected setting for SA, so one should intuitively expect better performance.

\item[(iv)] The theory is further substantiated with numerical experiments reported in Appendix A. The performance is benchmarked against standard bandit algorithms known to achieve optimal regret. 
\end{enumerate}

\textbf{Notation :} We have recalled all the notation used throughout the paper in Appendix B.

\section{Problem Formulation and the Algorithm}

The basic procedure of  simulated annealing consists of:
\begin{itemize}
\item A finite action/arm/solution set $\A$ with cardinality $|\A|=k$, to which is associated a real valued, bounded  cost function $\mu : \A \to \mathbb{R}$. The  action with the least cost is assumed to be unique and denoted by $a^*$.\footnote{This assumption is without any loss of generality for all the results in this work.} In the noisy setting, one does not know the exact value of $\mu_a$ and has only access to sample observations (drawn according to $P_{a}$). 

\item There is a bijective correspondence between the elements of $\A$, the set
of all the possible configurations of the optimization problem and the nodes of the graph. We denote the latter by $\G:=\{\A,\E\}$, where $\E$ is the set of edges. Accordingly, for each arm $a \in \A$, there is a set $\N(a) \subset \A / \{a\}$ called the neighbourhood set of $a$. It is assumed that $a \in \N(a') $ if and only if $a' \in \N(a)$. 

\item For any $a\in \A$, a collection of positive co-efficients $g(a,a'),\, a'\in \N(a)$ such that $$\frac{g(a,a')}{g(a)}  \leq 1, \,\,\,\text{ where }\,\,g(a):=\sum_{a' \in \N(a)} g(a, a').$$ Here, $\frac{g(a,a')}{g(a)}$ represents the probability of selecting a neighbouring candidate arm $a'$ for transition, when the present state of the algorithm is $a$. 

\item A cooling schedule $T: \mathbb{N} \to (0, \infty)$, assumed to be non increasing. $T(n)$ is also referred to as the temperature of the algorithm. 
\end{itemize}

Given these parameters, the SA algorithm consists of a discrete time inhomogeneous Markov chain, whose transition mechanism (for the noiseless case), denoted by $P(n)$ for temperature $T_n$, can be described as:

 \begin{equation}\label{MC}
  			 P_{x, y }(n) = \begin{cases}
   					 0, & \text{if $y \notin \N(x)$}\\
   					 \frac{g(x, y)}{g(x)} \exp\Big\{ \frac{- \big(\mu_y - \mu_x )^+\big)}{T_n} \Big\}, & \text{otherwise}
 					 \end{cases}
			\end{equation}
   and
           $$
           P_{x,x}(n) = 1 - \sum_{i \in \N(x)}  P_{x, i }(n), 
           $$
where $(x)^+:= \max(0,x)$. It is quite natural to assume a \textit{connected} graphical structure for modelling problems to be solved via SA (see Section 3, \cite{bert}). We recall that $g(a,a')/\sum_{i\in \N(a)}g(a,i)$ represents the probability of selecting the state $a' \in \N(a)$, given that the current state is $a$. A common practice is to set this equal to $\frac{1}{|\N(a)|}$ for all $a' \in \N(a)$. Let 
\begin{equation}\label{pi}
\pi_i(n) := \frac{g(i) \exp( -\frac{ \mu_i}{T_n} )}{ \sum_{j =1}^k g(j) \exp( -\frac{ \mu_j}{T_n} ) }. 
\end{equation}
We assume that
$$
g(a,a') = g(a',a) = 1
$$ 
whenever $a$ and $a'$ are neighbours. The above condition guarantees that (\ref{pi}) is the stationary distribution of $P(n)$, if we freeze the temperature at $T_n$ (see Proposition 3.1, \cite{mitra}). Additionally, this also makes sure that the detailed balance equation holds, so the time homogeneous Markov chain $P(n)$ is time reversible. We recall all the assumptions  taken thus far in the following: \\

\textbf{Assumptions:} 
\begin{enumerate}

\item[(i)] We assume that there exists a unique arm $a^*$, such that 
$$
a^* = \textrm{argmin}_{j \in \A} \,\,\, \mu_j.
$$

\item[(ii)]  The directed graph $\G$ is a connected. For all neighbours $a,a' \in \A$, we have $
g(a,a') = g(a',a).
$

\item[(iii)] The noise in the observations is assumed to be sub-Gaussian.  This means that for any $\epsilon>0$,
$$
\mathbb{P} (X \geq \epsilon)  \leq \exp \Big( -\frac{\epsilon^2}{2\sigma^2}\Big)
$$
for sone $\sigma > 0$, where $X$ is a generic noise variable. The family of sub-Gaussian bandits with $k$-arms having variance $\sigma^2$ is denoted by $\E^k_\text{SG}(\sigma^2)$.
\end{enumerate}

\subsection{An asymptotic perspective of SA with noisy observations}
It is instructive here, before we proceed to perform a theoretical analysis, to provide some intuition as to why SA with noisy observations should in principle work. Consider any two arms $i$ and $j \in \N(i)$ with distinct means $\mu_i<\mu_j$ and set $\Delta_{ij} :=  \mu_j -\mu_i$. Since we do not know $\mu_i$ and $\mu_j$, we will use the empirical mean estimates in their stead in the transition mechanism (\ref{MC}). Accordingly, let $\hat{\mu}_i(n)$ and $\hat{\mu}_{j}(n)$ be the calculated empirical mean at arm $i$ and $j$. Consider the event $G_n$ that $i$ and $j$ have been pulled $T_i(n)$ and $T_{j}(n)$ times respectively upto round $n$.  Then, since our arms are sub -Gaussian, the Cramer Chernoff method gives the following bound,
\begin{equation*}
	\mathbb{P} \Bigg(| \hat{\mu}_p(n) -\mu_p | \geq \frac{\Delta_{ij}}{2}\Bigg) \leq 2 e^{- \frac{T_l(n )  \Delta^2_{ij}}{16 \sigma^2}} \,\,\,\,\,\,\,\,\,\text{for }p=i,j,
\end{equation*}
where $T_l(n) =\min \{T_i(n),T_j(n)\} $. Suppose that $T_l(n) = \Omega\Big( \frac{16 \sigma^2 \log n }{ \Delta^2_{ij}}\Big)$, then  
\begin{equation*}
\mathbb{P} (G_n):=	\mathbb{P} \Bigg(| \hat{\mu}_p(n) -\mu_p |  \leq \frac{\Delta_{ij}}{2} \Bigg) \geq 1- \frac{2}{n}  \ \,\,\,\,\,\,\,\,\,\text{for }p=i,j.
\end{equation*}
Then, for event $G_n$, which occurs with high probability for large $n$, we have
\begin{align*}
 P(a_{n+1} = j | a_{n} =i  ) & = \exp\bigg(-\frac{(\hat{\mu}_j- \hat{\mu}_i)^+}{T_n}  \bigg\}  \Bigg)\\
 &\sim \exp\bigg(-\frac{(\mu_j - \frac{\Delta_{ij}}{2} - \mu_i+  \frac{\Delta_{ij}}{2}}{T_n}   \bigg) \\
&\sim \exp\bigg(-\frac{(\mu_j - \mu_i)}{T_n}  \bigg). 
\end{align*}
Conversely,  $P(a_{n+1} = i | a_{n} = j  ) =1$ with high probability. Thus, for fixed time horizon $n$, $\Omega(\log n)$ pulls for each arm appear to be sufficient to distinguish the arms with high probability and one could  guess that the convergence behaviour of SA would be somewhat similar to a noiseless setting. Admittedly, while this two arm bandit instance may only present a highly simplified view of the task at hand, the above facts do to some extent (especially for the standard MAB)  underpin the main aspects of the theory. It is not surprising the original study by \cite{gelfand} came to the conclusion that the post sampling noise variance should decrease at a minimum rate of $o( \log n )$ (see proposition 3.1 in \cite{gelfand}) for SA to succeed despite noisy observations.
\begin{figure}[h]
      \includegraphics[width=145mm,height=60mm]{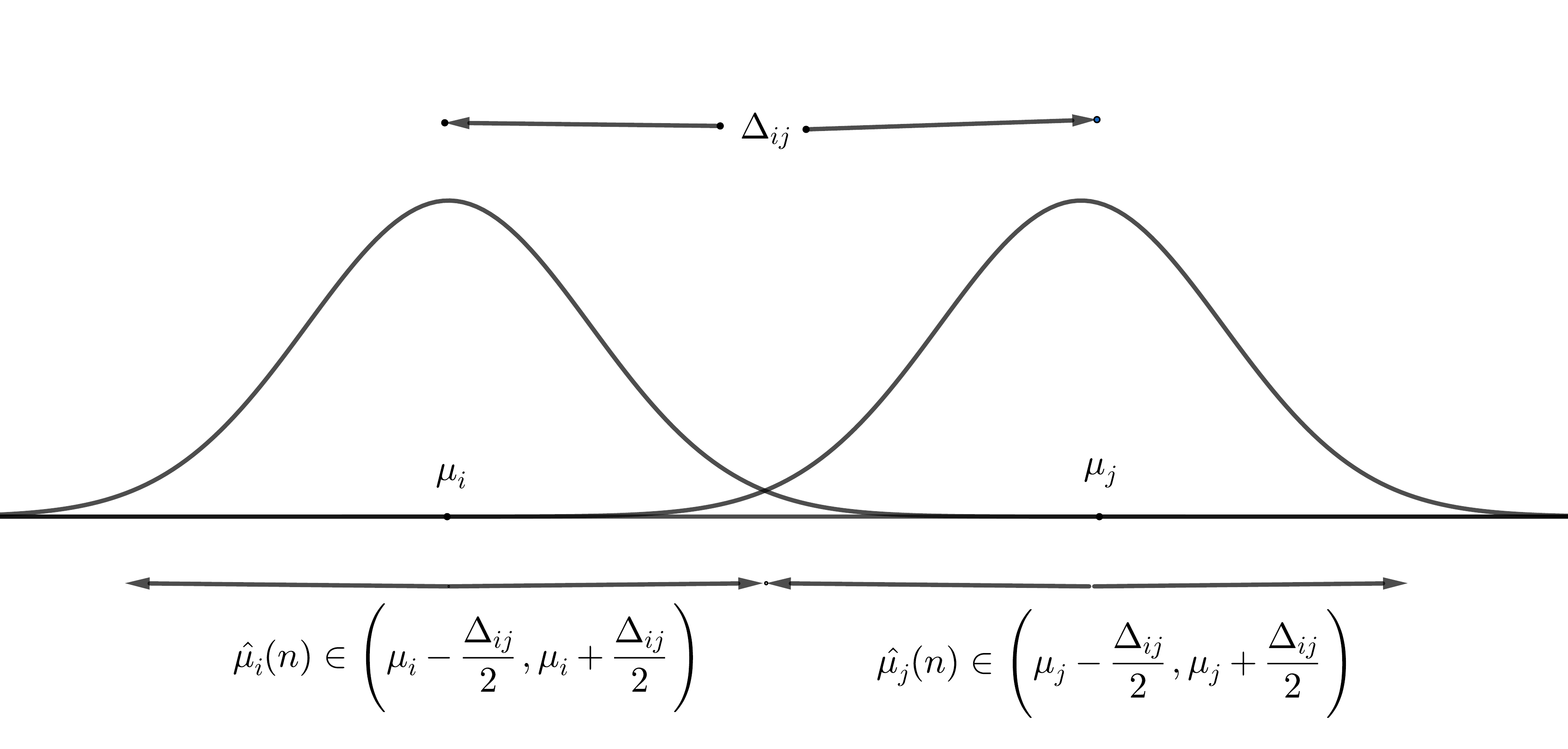}
      \caption{The event $\hat{\mu}_p(n) \in [\mu_p+\Delta_{ij}/2, \mu_p+\Delta_{ij}/2 ]$, $p=i,j$,  occurs with high probability for $T_p(n) = \log n$ }.
      \label{collapse}
 \end{figure}

The SA algorithm can be thought of as a (time inhomogeneous) random walk on $\A$, which gradually gets reinforced, via the Glauber dynamics, to prefer the low energy (mean) states (arms). So, it lends itself naturally to be analyzed as a vertex reinforced random walk (VRRW) \cite{pem}, \cite{benaim}. In what follows, we look at its workings through the lens of stochastic approximation. This will help us relate its asymptotic behaviour to the widely studied field of evolutionary games which may have broader implications for analysis of algorithms besides SA. Let $\{a_n\}_{n\geq0}$ denote the time inhomogeneous Markov chain at hand.  Also, let $a_n \in \A_n$ denote the value of the arm pulled at time $n$ and $z_i(n)$ denote  the total loss incurred by taking action $i$ up to time $n$, so that
$$
z_i(n) = \sum_{m=1}^n \mathbb{I} \{a_m = i\} \tilde{\mu}_{i}.
$$
where $\tilde{\mu}_{i} = \mu_{i} + \xi_{i}$ with $\xi_{i}$ being sub-Gaussian. Let $x(n)$ denote the empirical frequency, i.e. the $i$'\textit{th} component of $x(n)$ is the fraction of the total loss till time $n$, so that
$$
x(n) = \frac{z(n)}{n},
$$
\textit{We analyze the frequency with which any arm is visited  by studying $x(n)$ as a stochastic approximation scheme}. Accordingly, we re-write the above iteration component-wise as:
\begin{equation}\label{SA-d}
x_i(n) = \Big(1 -\frac{1}{n}\Big)x_i(n-1) + \frac{  \mathbb{I} \{a_{n-1} = i\} \tilde{\mu}_{i} }{n}
\end{equation}
or equivalently,
$$
x_i(n) = \Big(1 -\frac{1}{n}\Big)x_i(n-1) + \frac{1}{n} \bigg( \mathbb{I} (a_{n-1} = i) \mu_i  + M_{n} \bigg),
$$
where 
\begin{align*}
M_n &:= \mathbb{I} \{a_{n-1} = i\}  \tilde{\mu}_{i}  - \mathbb{I} (a_{n-1} = i ) \mu_{i}  \\&= \mathbb{I} \{a_{n-1} = i\} (   \tilde{\mu}_{i} - \mu_{i} ) 
\end{align*}
is a Martingale difference sequence. The probability transition mechanism in (\ref{SA-d}) is given by,
\begin{equation}\label{MC1}
\mathbb{P} (a_{n+1} = j | a_n=i) = \mathbb{I} \big\{j \in \N(a_n)\big\}\,\, \Bigg( \frac{g(i,j)}{g(i)}  \exp\bigg(-\frac{( \hat{\mu}_j - \hat{\mu}_i )^+}{T_n} \bigg)\Bigg).
\end{equation}
The scaled version of the previous iteration can be written as follows:
\begin{equation}\label{y}
y_i(n) = \Big(1 -\frac{1}{n}\Big)y_i(n-1) + \frac{1}{n} \bigg( \mathbb{I} (a_{n-1} = i)   +  \bar{M}_i (n) \bigg),
\end{equation}
where $y_i := \frac{x_i}{\mu_i} $ and $\bar{M}_i (n) : = \frac{M^i_n}{\mu_i}$. Let 
\begin{equation*}
\epsilon_n := \mathbb{P} (a_{n} = j | a_{n-1} =i) - \bar{\mathbb{P}} (\bar{a}_{n} = j | \bar{a}_{n-1}=i),
\end{equation*}
where $\bar{\mathbb{P}} (\bar{a}_{n} = j | \bar{a}_{n-1}=i)$ is given by (\ref{MC}) and $\{\bar{a}_n\}$ is the process generated by the transition mechanism $\bar{\mathbb{P}}$. To be more precise,  $\bar{\mathbb{P}} (\bar{a}_{n} = j | \bar{a}_{n-1}=i)$ is (\ref{MC1}) with noiseless estimates (i.e., $\hat{\mu}_i = \mu_i$, for all $i \in\A$). So, (\ref{y}) can be written as:
\begin{equation*}
y_i(n) = \Big(1 -\frac{1}{n}\Big)y_i(n-1) + \frac{1}{n} \bigg( \mathbb{I} (\bar{a}_{n-1} = i) + \epsilon_n  +  \bar{M}_i (n) \bigg).
\end{equation*}
The difference of the above equation from (\ref{y}) is that (provided $\epsilon_n \to 0$, see section 4.1 for a bound on this quantity), the probability of transition is governed by $\bar{\mathbb{P}}(\cdot|\cdot)$ instead of  $\mathbb{P}(\cdot|\cdot)$. For SA to succeed, we should have $y(n) \to y^*$, where 
\begin{equation}\label{truth}
y^*_{a^*} =1 \,\,\,\,\text{  and  }\,\,\,\, y^*_{a} =0 \,\,\forall a\neq a^*.
\end{equation}
For a fixed $T$, by the standard analysis of stochastic approximation with Markov noise (\cite{borkar}, Chapter 6), the sequence ${y(n)}$ will
almost surely track the asymptotic behavior of the o.d.e., \footnote{Actually, the time schedule does not have to be constant, the claim will hold for $T(n) =\frac{\gamma}{\log n}$.}
\begin{equation}\label{repl}
\dot{y}(t) = \pi_T(y(t))-y(t), 
\end{equation}
where $\pi_T(y)$ is the $y$-dependant stationary distribution of time homogeneous chain (\ref{MC}). The dynamics (\ref{repl}) is a special case of the replicator dynamics of evolutionary biology \cite{hof}, \cite{sand}, evolving on the unit simplex $\Delta^{k}$. Plugging the value of (quasi) stationary distribution for (\ref{MC}) in the above shows that any equilibrium point, say $y^*$, satisfies:
$$
y^*_i \propto \frac{\exp\Big( -\frac{\mu_i}{T} \Big)}{\sum_{j \in \A} \exp\Big( -\frac{\mu_j}{T} \Big)}.
$$
Letting $T\to 0$ will lead to (\ref{truth}). This again leads us to the conclusion that if the noise variance goes to zero (i.e. $\epsilon_n \to 0$ here), the noisy observations will not affect the algorithm. 

\subsection{Failure of Greedy Strategy:}

 In this subsection, we show by an example that the standard $\epsilon$-greedy algorithm may lead to a sub-optimal choice when the action sets are graphically constrained\footnote{We remark that the arguments presented can be easily extended to other exploration algorithms like the Boltzman exploration.} (see also Section 5, \cite{avr}). We do this by considering the bandit instance, $\mu := (1-\delta , 1+\delta, 1+2\delta , 1)$ with $0<\delta <1$, where the agents are connected via a linear graph.  We will show that for the $\epsilon$-greedy heuristic, there is a strictly positive probability of getting stuck in a sub-optimal choice (i.e. agent 4 with loss $1$) and thus the regret scales linearly with the time horizon $T$. 
For our purposes, we can assume deterministic rewards without any loss of generality. One can easily deduce that the transition matrix $P$ for the $\epsilon$-greedy strategy, under such conditions, can be written as:
$$
P := (1-\epsilon)   \underbrace{
\begin{bmatrix}
1& 0 &0 &0 \\
1 & 0 & 0 &0\\
0&0 & 0 & 1 \\
0&0 & 0 & 1 \\
\end{bmatrix}}_{A \rightarrow \text{Greedy selection}}
+ \epsilon
 \underbrace{\begin{bmatrix}
\frac{1}{2} & \frac{1}{2} & 0 &0 \\
\frac{1}{3} & \frac{1}{3} & \frac{1}{3} &0\\
0&\frac{1}{3} & \frac{1}{3} & \frac{1}{3}\\
0& 0 & \frac{1}{2} & \frac{1}{2}\\
\end{bmatrix}}_{B\rightarrow \epsilon\text{-exploration}}
$$
The chain, without exploration, moves deterministically to the neighbour
(including itself) with the lower loss. For $\epsilon \in (0,1)$, the $\epsilon$-greedy policy can easily be proved to have a stationary distribution that concentrates equally on $\{1,4\}$ by the symmetry of the problem which shows a positive probability of getting stuck in a sub-optimal choice. More concretely, suppose that the initial probability distribution over the states is given by $\nu(0)=(0,0,0,1)$ and the exploration factor $\epsilon :=\frac{1}{T}$, where $T$ is the time horizon.. Then, the distribution at time $n$ denoted by  $\nu(n)$ can be written as
\begin{align*}
\nu(n) &= \nu(0) P^n \nonumber \\
&=\nu(0)\begin{bmatrix}
\text{*} & \text{*} &\text{*} &\text{*} \\
\text{*} & \text{*} & \text{*} &\text{*}\\
\text{*} & \text{*} & \text{*} &\text{*} \\
\text{*} & \text{*} & \text{*} & 1 -\frac{\epsilon}{2} \\
\end{bmatrix}^n
\end{align*}
 Let $\nu_i(n)$ for $1 \leq n \leq T$ denote the probability of picking state $i$ at step $n$. Then with $\epsilon = \frac{1}{T}$, we have
\begin{equation}
\label{greedy}
\nu_4(n) \geq \Big(1-\frac{\epsilon}{2}  \Big)^n
\geq e^{-\frac{n\epsilon}{2-\epsilon}} \geq \frac{1}{e},
\end{equation}
where we have used the fact that
$$
(1-x) \geq e^{-\frac{x}{1-x}}, \text{ for } x \in (0,1)
$$
in the second inequality and that $n\epsilon \leq 1$ for $n \le T$ in the third inequality.
The last inequality in (\ref{greedy}) gives a strictly positive lower bound on the probability
independent of $T$. Contrast this with the standart results of $\epsilon$-greedy policy where in the absence of graphical constraints, optimality is achieved is with the same time step.
%
%



\section{Noiseless Case:}
In this section we perform a regret analysis of the simulated annealing algorithm with exact observations. We first briefly review the mode of convergence for time inhomogeneous Markov Chains and introduce some theoretical constructs that will also be of use to us later. We presently recall the notion of weak ergodic convergence. Let
\begin{equation*}
  \PPP(m,n) := \prod_{t=m}^{n -1}P(t) 
\end{equation*}
for any time inhomogeneous chain with transition mechanism $\{P(t)\}_{t\geq0}$.

\bigskip

\textbf{Definition 1} (\textit{Weak Ergodic Convergence}) A time-inhomogeneous Markov chain is said to be weakly ergodic if, for all $m$,\footnote{$\|x\|$ is always the $\ell_1$ norm and $\|A\|$ the infinity norm in this paper (see Appendix B for definitions). }
$$
\lim_{n \to \infty} \sup_{\nu(0),\tilde{\nu}(0)}  \| \nu(0) \PPP(m,n)-\tilde{\nu}(0) \PPP(m,n)\| =0.
$$
The above definition implies a tendency towards equality of the rows of $P(m, n)$, i.e., a `loss of memory' of the initial conditions. The investigation of conditions under which weak ergodicity holds is aided by the introduction of the following quantity:

\bigskip

\textbf{Definition 2} (\textit{Coefficient of Ergodicity})  Given a stochastic matrix $P$, its coefficient of ergodicity, denoted by $\kappa(P)$, is defined to be :
$$
\kappa(P) = \frac{1}{2}\max_{i,j}  \sum_{s=1}^k |P_{is} - P_{js}| = 1- \min_{i,j} \sum_{s=1}^k \min \{P_{is},P_{js}\}.
$$
We recommend the reader to \cite{issac} or \cite{seneta} for a detailed discussion of the importance of using ergodicity coefficients to study time inhomogeneous Markov chains. The main two properties which we will use are (Theorem V.2.4, \cite{issac}):
\begin{equation}\label{prop}
\kappa(PQ) \leq  \kappa(P)\kappa(Q) \,\,\text{ and }\,\, \|  PQ\| \leq \|P\| \kappa(Q)
\end{equation}
for any stochastic matrices $P$, $Q$. We next define the concept of critical depth,  first introduced in \cite{hajek}:

\bigskip

 \textbf{Definition 3} (Critical Depth) Let $p= (i_0,\cdots,i_n)$ be any path in $\G$, i.e., for any $i_p \in p$, $i_{p+1} \in \N(i_p)$. Let
\begin{equation*}
\text{Elev}(p) := \max_{0 \leq m \leq n} \mu_{i_m}\,\,\,\,\text{and}\,\,\,\,
e(p) := \text{Elev}(p) -\mu_{i_0} -\mu_{i_n}.
\end{equation*}
Let $\mathcal{E}(p) = \sup_{p\in \mathcal{P}} e(p) $, for  any choice of path $\mathcal{P}$. Also, let $\textbf{e} $ be the minimum value of $\mathcal{E}(p)$ as $\mathcal{P}$ runs over all selections of allowable paths. Then, the \textbf{critical depth}, denoted by $\gamma^*$, is defined to be 
$
\gamma^* :=  \textbf{e}  + \mu_{a^*}.
$
\\

The formal statement of critical depth is a bit hard to digest. A high level intuition for $\gamma^*$ is that it can be thought of as the least upper bound on the energy barrier that one has to climb in order to reach the optimal arm. To be more precise, let us say that arm $i$ communicates with the optimal arm $a^*$ at height $h$, if there exists a path starting at $i$ and ending at $a^*$ and such that the largest value of any arm encountered along the path is $\mu_i+ h$. Then, $\gamma^*$ is the smallest number such that every $i\in \A$ communicates with $a^*$ at height $\gamma^*$.  We will later see how the ergodicity co-efficient is related to the critical depth (see eq. (\ref{tttt})). Let
$$L:= \max_{i\in \A}\max_{j \in \N(i)} |\mu_j - \mu_i| \,\,\,\text{ and }\,\,\, R:= \min_{i\in \A}\min_{j\in \N(i)} \frac{g(i,j)}{g(i)}$$ 
We have for any $i$ and $j \in \N(i)$ in (\ref{MC}),
\begin{equation}\label{stau}
	P_{ij}(m) \geq R e^{-\frac{L}{T_m}} .
\end{equation}
From the definition of ergodicity coefficient,  we have 
\begin{align*}
\kappa(\PPP(m-\tau,m) ) &\leq   1 - \min_{i,j} \Big\{ \sum_{s=1}^k \min \Big(\PPP_{is}(m-\tau,m), \PPP_{js}(m-\tau,m)\Big) \Big\} \\
&\leq 1 - \max_s \min_i \PPP_{is}(m-\tau,m).
\end{align*}
The second inequality follows from eq. (4.7), Section 4.3, \cite{seneta}. Let $\tau$ be the smallest integer such that $ \max_s \min_i P_{is}(m-\tau,m)  >0$. Using (\ref{stau}) in the above we have:
\begin{equation}\label{kappa}
\kappa(\PPP(m-\tau,m) ) \leq 1- R^{\tau} \exp(- \frac{\tau L}{T_{m-1}}).
\end{equation}
If for $T_m := \frac{\gamma}{\log m}$, the condition $\frac{\tau L}{\gamma} \leq 1$ is satisfied, then the chain corresponding to (\ref{MC}) is \textit{weakly ergodic} (see Theorem 5.1, \cite{mitra}). Our aim will be to get an upper bound on how fast such a convergence takes place and its dependence on $\gamma$, by establishing a regret bound. To do so, we first prove a preliminary result to bound the probability of selecting a sub-optimal arm.  
\begin{lem} 
Let the cooling schedule be of the parametric form $T_t =\frac{ \gamma }{\log   (t+1)}$. Then, there exists a constant $m_0 <\infty$ such that the probability of selecting a sub-optimal arm for any $t\geq m_0$ is bounded by
\begin{equation}\label{1-term-1}
\mathbb{P}(a_t \in \A/ a^*) \leq   \frac{2 \exp\Big(  \frac{R^\tau }{\tau^{\frac{\tau L}{\gamma}}(1-\frac{\tau L}{\gamma})} (\frac{m_0}{\tau}+1)^{1-\frac{\tau L}{\gamma}} \Big)}{ \exp\Big(  \frac{R^\tau}{\tau(1-\frac{\tau L}{\gamma})}  t^{1- \frac{\tau L}{\gamma}}  \Big)} + \sum_{a\in\A/a^*} \frac{4g(a)}{g(a^*)} (t+1)^{-\frac{\mu_a -\mu_{a^*} }{\gamma}},
\end{equation}
so that,
$$
\mathbb{P}(a_t \in \A/ a^*)  = \mathcal{O} \Bigg(\exp\Big( - \frac{R^\tau}{\tau(1-\frac{\tau L}{\gamma})}  t^{1- \frac{\tau  L}{\gamma}}  \Big)+  \sum_{a \in \A/a^*}\frac{1}{ (t+t_0)^{\frac{\mu_a -\mu_{a^*} } {\gamma}}}\Bigg).
$$
\end{lem} 
\begin{proof}
Let $\nu(0)$ denote the initial probability distribution. The probability of selecting a bad arm  can be written as:
\begin{align*}
\mathbb{P}(a_t \in \A/ a^*) &= \sum_{i \in \A/a^*} [\nu(0)\PPP(0,t)]_i ,\\
&\leq \|  \nu(0)\PPP(0,t) -e^* \|, 
\end{align*}
where $e^*$ is the unit vector with all entries zero expect the one corresponding to $a^*$. To prove the result we upper bound the term $ \| \nu(0)\PPP(0,t) -e^*\|$. Accordingly, we first decompose this quantity as:
\begin{equation}\label{main-dec}
\| \nu(0)\PPP(0,t) -e^*\| \ \leq\underbrace{\|\nu(0)\PPP(0,t) - \nu(0)Q(t) \|}_{\text{Term 1}} +\underbrace{ \| \nu(0)Q(t)- e^*\|}_{\text{Term II}},
\end{equation}
where :
\[
Q(t) =
\left[
  \begin{array}{ccc}
    \horzbar & \pi(t) & \horzbar \\
             & \vdots    &          \\
    \horzbar & \pi(t) & \horzbar
  \end{array}
\right]
\]
with $\pi(t)$ being the quasi-stationary distribution defined in (\ref{pi}).

\textbf{\\Term I: } We have,
\begin{align*}
\|\nu(0)\PPP(0,t) - \nu(0)Q(t) \|    &\leq   \| \PPP(0,t) -  Q(t) \|.
\end{align*}  
Furthermore, we have for any $0<\m<t$, the following decomposition
\begin{equation}\label{dec-2}
 \| \PPP(0,t) -  Q(t) \|  \leq   \| \PPP(0,t)  -   Q(\m) \PPP(\m,t) \|+     \|  Q(\m)\PPP(\m,t)  - Q(t) \|.
\end{equation} 
We deal with each of these terms separately. Consider the first term:
\begin{align*}
 \| \PPP(0,t)  -    Q(\m) \PPP(\m,t) \| &\leq \| \PPP(0,\m)-Q(\m)\|  \kappa(\PPP(\m,t)) \\
& \leq  2 \prod_{k=t_0}^{\frac{t}{\tau}-1} \kappa(\PPP(k\tau,(k+1)\tau)) 
\end{align*}\label{term1'}
for some $t_0 \tau = \m$, assuming $t_0 \in \mathbb{Z}^+$ w.l.o.g. Since $\kappa(\PPP(m-\tau,m)) \leq 1 - R^\tau \exp\Big( -\frac{\tau L}{T_{m-1}}\Big)$ (see eq. (\ref{kappa})), we get
\begin{align}\label{term1'}
 \| \PPP(0,t)  -    Q(\m) \PPP(\m,t) \| &\leq 2 \prod_{k=t_0}^{\frac{t}{\tau}-1}  \Big[1 - R^\tau \exp\Big( -\frac{\tau L}{T_{k\tau+\tau-1}}\Big)\Big] \nonumber \\
&=  2 \prod_{k=t_0}^{\frac{t}{\tau}-1}  \Big[1 -  R^\tau \exp\Big(-\frac{\tau  L}{\gamma }\log (k\tau+\tau) \Big) \Big] \nonumber  \\
&=  2 \prod_{k=t_0}^{\frac{t}{\tau}-1}  \Big[1 -   \frac{R^{\tau}}{(k\tau+\tau)^\frac{\tau  L}{\gamma}} \Big] \nonumber  \\
&\leq    2\exp\Big( -R^\tau \sum_{k=t_0}^{\frac{t}{\tau}-1}\frac{1}{\big((k+1)\tau\big)^\frac{\tau  L}{\gamma}} \Big) \nonumber  \\
&=  2\exp\Big( - \frac{R^\tau}{\tau^{\frac{\tau  L}{\gamma}}}  \sum_{k=t_0}^{\frac{t}{\tau}-1}\frac{1}{\big(k+1\big)^\frac{\tau  L}{\gamma}} \Big) \nonumber  \\
&\leq 2 \exp\Big( - \frac{R^\tau}{\tau^{\frac{\tau  L}{\gamma}}(1-\frac{\tau L}{\gamma})}  \Big(  \Big(\frac{t}{\tau}\Big)^{1-\frac{\tau L}{\gamma}}- (t_0+1)^{1-\frac{\tau L}{\gamma}} \Big) \Big)  \nonumber  \\
 \| \PPP(0,t)  -    Q(\m) \PPP(\m,t) \| &\leq \frac{2 \exp\Big(  \frac{R^\tau }{\tau^{\frac{\tau  L}{\gamma}}(1-\frac{\tau L}{\gamma})} (t_0+1)^{1-\frac{\tau L}{\gamma}} \Big)}{ \exp\Big(  \frac{R^\tau}{\tau(1-\frac{\tau L}{\gamma})}  t^{1- \frac{\tau  L}{\gamma}}  \Big)}.
\end{align}
We now consider the second term in (\ref{dec-2}). We note that stationarity of $\pi_t$ implies $Q(t)P(t) = Q(t)$. So,
\begin{multline*}
Q(l)\PPP(l,m) = Q(l) \PPP(l,m) -Q(l+1) \PPP(l+1,m) + Q(l+1) \PPP(l+1,m)\\
= \big( Q(l) -Q(l+1) \big)\PPP(l+1,m)  +Q(l+1)\PPP(l+1,m) .
\end{multline*}
Recursion on the above equation gives
\begin{align*}
Q(l)\PPP(l,m) &= \big( Q(l) -Q(l+1) \big)\PPP(l+1,m)  + \big( Q(l+1) -Q(l+2) \big)\PPP(l+2,m)\\
& \,\,\,\,\,\,\,\,\,\,\,\,\,\,\,+Q(l+2)\PPP(l+2,m) \\
&= \sum_{t=l}^{m-1} \big(Q(t) -Q(t+1)\big)  +Q(m).
\end{align*}
Take $l=\m$ and $m=t$ in the above to get:
\begin{align*}
\| Q(\m)\PPP(\m,t) -Q(t) \| &\leq \sum_{t=\m}^{t-1} \| Q(t) -Q(t+1)\|. 
\end{align*}
To bound this we have to study the decay of the  transition probabilities w.r.t cooling schedule. Accordingly, one can verify the following relation in a straightforward manner:
\begin{multline*}
\frac{g(a)T^2}{\pi^2_a(t)} \dot{\pi_a}(T) = - \sum_{a' \in \A} g(a')(\mu_{a'}  - \mu_{a} ) \exp\Big(- \frac{\mu_{a'}  - \mu_{a} }{T} \Big)  
\\=   \sum_{a' : \mu_{a'} < \mu_{a}} g(a')(\mu_{a}  - \mu_{a'} ) \exp\Big( \frac{\mu_{a}  - \mu_{a'} }{T} \Big)- \sum_{a' : \mu_{a'} > \mu_{a}} g(a')(\mu_{a'}  - \mu_{a} ) \exp\Big(- \frac{\mu_{a'}  - \mu_{a} }{T} \Big).
\end{multline*} 
 For $a =a^*$, it can be easily seen that  $\dot{\pi_i}(T) < 0$, so that $\pi_{a^*} (t+1) > \pi_{a^*} (t)$ (since temperature $T$ is a decreasing function of $t$ ). For $a \in \A/a^*$, we have $\lim_{T \to 0} \dot{\pi}(T) > 0$. Thus, there exists a $m_0<\infty$, such that 
$$
\dot{\pi}_a(t) > 0,\,\,\,\,\, \forall t\geq m_0, \,a \in \A/a^*   
$$    
which implies $\pi_{a}(t+1) < \pi_{a}(t)$, for all $t \geq m_0$. Then, for all $t\geq m_0$
\begin{align*}
\| \pi(t)- \pi(t+1)\| &=   \sum_{a \in \A/a^*}  |\pi_{a}(t)- \pi_a(t+1)| +  | \pi_{a^*}( t ) -  \pi_{a^*}( t+1) |\\
& =   \sum_{a \in \A/a^*}  \pi_{a}(t)- \pi_a(t+1) +  \pi_{a^*}( t + 1 ) -  \pi_{a^*}( t) \\
& = 2 \sum_{a \in \A/a^*}  \pi_{a}(t+1)- \pi_a(t),
\end{align*}
where we have used the fact that $ \pi_{a^*}(t)  = 1 - \sum_{a \in \A/a^*}\pi_{a}(t) $. Thus, we have, with $\m = m_0$, the following bound 
\begin{equation}\label{alt-bound}
\sum_{t= m_0}^{t-1} \| \pi(t+1)- \pi(t) \| \leq 2 \sum_{t =  m_0}^{t-1}    \sum_{a \in \A/a^*} \pi_{a}(t+1)- \pi_a(t) \leq   2 \sum_{a \in \A/a^*} \pi_a(t).
\end{equation}
The above term will be bounded in the same way as Term II as follows:

\textbf{\\Term II :} We have
\begin{align*}
\|\nu(0)Q(t) -e^*\|   & \leq \| \pi(t) -e^*\|\\
& = 1 - \pi_{a^*}( t) + \sum_{a \in \A/a^*} \pi_a(t) \\
&= 2 \sum_{a \in \A/a^*} \pi_a(t).
\end{align*}
Recalling the the definition of $\pi_a(t)$, we have
\begin{align}\label{term2}
\|\nu(0)Q(t) -e^*\|   & \leq  2\frac{\sum_{a \in \A/a^*} g(a)\exp\Big({-\frac{\mu_a}{T_t}}\Big) }{\sum_{a \in \A } g(a) \exp\Big({-\frac{\mu_a}{T_t}}\Big)} \nonumber \\
 & = 2 \frac{\sum_{a\in\A/a^*} \frac{g(a)}{g(a^*)}\exp\Big({-\frac{(\mu_a - \mu_{a^*})}{T_t}} \Big) } {1+  \sum_{a \in \A/a^* } \frac{g(a)}{g(a^*)} \exp\Big({-\frac{(\mu_a - \mu_{a^*})}{T_t}} \Big)}\nonumber\\
 & \leq 2\sum_{a\in\A/a^*} \frac{g(a)}{g(a^*)} (t + 1)^{-\frac{\mu_a -\mu_{a^*} }{\gamma}}.
\end{align}
To conclude the proof of  the lemma, bound Term I  in (\ref{main-dec}) using (\ref{term1'}) with $m_0 = \m  = t_0\tau$, and (\ref{alt-bound}) in (\ref{dec-2}). Term II has been bounded in (\ref{term2}).
\end{proof}
We recall that $\tau$ is the smallest integer such that $ \max_s \min_i P_{is}(m-\tau,m)  >0$. It was proved in \cite{chiang} using the results of \cite{vent} (see also Proposition 7.2, \cite{tsitsurvey}) that the smallest such $\tau$ satisfies:
\begin{equation}\label{tttt}
\gamma^* \leq \tau L
\end{equation}
Thus, we have for  $\gamma > \tau L \geq \gamma^*$,
$$
0 < 1-\frac{\tau L}{\gamma} \leq 1 - \frac{\gamma^*}{\gamma}
$$
to ensure the decay of the first term in (\ref{1-term-1}). Thus, we recover the characterization of the cooling schedule of \cite{hajek} for finite time convergence bounds. We can use Lemma 1 to upper bound the regret.

\begin{thm}
Let $\Delta:=  \max_{i \in \A/a^*}\{ \mu_{i} - \mu_{a^*}\}$ and $\gamma >\tau L\geq \gamma^*$. The regret of SA is bounded by:
\begin{multline*}
R_n \leq   \frac{2 \Delta \exp\Big(  \frac{R^\tau }{\tau^{\frac{\tau L}{\gamma}}(1-\frac{\tau L}{\gamma})} (\frac{m_0}{\tau}+1)^{1-\frac{\tau L}{\gamma}} \Big)}{  \Big( \frac{R^\tau}{\tau(1-\frac{\tau L}{\gamma})}\Big)^{\big\lfloor \frac{1}{1 - \frac{\tau L}{\gamma}}\big\rfloor }} \Gamma\Bigg( 1+ \Bigg\lceil \frac{1}{1 - \frac{\tau L}{\gamma}}\Bigg\rceil \Bigg) 
+\sum_{a\in\A/a^*} \frac{4 \Delta g(a)}{g(a^*)(1-\frac{\mu_a -\mu_{a^*} }{\gamma} )} (n+1)^{1-\frac{\mu_a -\mu_{a^*} }{\gamma}},
\end{multline*} 
for $n \geq m_0$, where $\Gamma$ is the gamma function. So, we have,
$$
R_n = \mathcal{O}\big( n^{1-\frac{\mu_a -\mu_{a^*} }{\gamma}}\big) 
$$
\end{thm}

\begin{proof}
Let $\Delta_i := \mu_i - \mu_{a^*}$. We recall the standard regret decomposition identity,
$$
R_n = \sum_{i=1}^k \Delta_i \EX[T_i(n)] = \sum_{i=1}^k \Delta_i \big(\sum_{t=1}^n \PP(a_t=i)\big).
$$ 
The regret can thus be bounded by (since $\Delta_{a^*}=0$)
$$
R_n \leq \Delta \sum_{t=1}^n P(a_t \in \A/\{a^*\}).
$$
The result is proved by using the previous lemma. The bound for the second term is obvious while the first term can be handled by noting that it can bounded by an integral of the form $\int_0^\infty e^{-cx^{a}} dx $ (with $a = 1 - \frac{\tau L}{\gamma}$ and $c =\frac{R^\tau}{\tau(1-\frac{\tau L}{\gamma})}$), which in turn can be bounded as: 
\begin{align*}
\int_0^\infty e^{-cx^{a}} dx &= \frac{1}{c^{\frac{1}{a}}a} \int_0^\infty u^{\frac{1}{a} -1}e^{-u}\,du \,\,\,\,\,\,\,\,\,(\text{set }cx^a = u)\\
&= \frac{1}{c^{\frac{1}{a}}a} \Gamma \Big(\frac{1}{a} \Big) \\
&= \frac{1}{c^{\frac{1}{a}}}  \Gamma\Big(1 + \frac{1}{a} \Big).
\end{align*}

\end{proof}

 \section{Noisy Simulated Annealing}
In this section we analyze SA with noisy observations in full generality. Accordingly, we only assume a connected graph (see assumption(ii)) in place of the fully connected setting of the previous section. The formal procedure is detailed in Algorithm 1.

\begin{algorithm}[htb]
\textbf{Input:} Arm set $\A = [k]$;  Graph $\mathcal{G} =\{\A, \,\E\}$;  Time horizon $T_{\text{hor}}$; Temperature $T_t$;\\
\textbf{Initialization:} $a_0 \in \A$, $\mu_{i} (0) = -\infty$ for all $i \in \A$.

\For{ $t=0,\cdots,T_{\text{hor}}$  } { 
            \textbf{(a)} At time $n$, select $a_{n+1} $ according to:
  		   $$
  			 P_{a_n, a_{n+1} } = \begin{cases}
   					 0, & \text{if $a_{n+1} \notin \N(a_n)$}\\
   					(1-\epsilon_n)  \frac{g(a_{n+1}, a_n)}{g(a_n)} \exp\Big\{ \frac{- \big(\tilde{\mu}_{a_{n+1}}(n) - \tilde{\mu}_{a_n} 	(n) )^+\big)}{T_n} \Big\} +  \frac{g(a_{n+1}, a_n)}{g(a_n)} \epsilon_n, & \text{otherwise}.
 					 \end{cases}
			$$
           and
           $$
           P_{a_n,a_n} = 1 - \sum_{i \in \N(a_n)}  P_{a_n, i }. 
           $$
          \textbf{(b)} Pull $a_{n+1}$ to get estimate $\mu_{a_{n+1}}'$ and update  the empirical mean according to: $$\tilde{\mu}_{a_{n+1}} (n+1) := \Big(1- \frac{1}{T_{a_{n+1}}(n)} \Big)\tilde{\mu}_{a_{n+1}} (n) + \frac{1}{T_{a_{n+1}}(n)} \mu_{a_{n+1}}',  $$
          where $T_{a}(n): =$ total number of times arm $a$ has been pulled till $n$.  For $a \neq a_{n+1}$, keep  
           $$
           \tilde{\mu}_{a}(n+1) = \tilde{\mu}_{a}(n).
           $$
		 
		 \textbf{(c)} $T_{n+1} = \textbf{Update} (T_n)$.

 }
 
 \textbf{Output}: Resulting Policy: $\nu(n)$ \\

\caption{ Simulated Annealing with noisy estimates}
\end{algorithm}

We will specify the conditions on the sequence $\{\epsilon_n\}$, which acts as an additional noise to the system, later. In particular,  $\epsilon_n \to 0$ and is required to make sure that $T_i(n) \to \infty$ for all $i$ which in turn implies $\hat{\mu}_i(n) \to \mu_i$ w.p.1. Our analysis builds upon the framework established in \cite{strockbook} to study SA.

\subsection{Noise Perturbation Bounds}

To establish a regret bound, we first study how noisy observations affect the transition probability. To do this, we construct a related stochastic process $\{\As_n\}_{n\geq 1}$, undisturbed by any noise so that to select the next arm, we use the actual mean difference. To be more precise, we recall the definition of the probability matrix of noiseless SA:
\begin{equation*}
  			\Ap_{a, a' }(n):= \Ap (\bar{A}_{n+1} = a' | \bar{A}_{n}= a)   
  			 := \begin{cases}
   					 0, & \text{if $a' \notin \N(a),\,a' \neq a$}.\\
   					 \frac{g(a, a')}{g(a)} \exp\Big\{ -\frac{ (\mu_{a'}-\mu_{a} )^+}{T_n} \Big\} , & \text{otherwise},
 					 \end{cases}
			\end{equation*}
for any $a,\,a' \in \A$. Our aim is to bound the difference:
$$
\xi_n := \Ap (\bar{A}_{n+1} = a | \bar{A}_{n}= a') - P (A_{n+1} = a | A_{n}=a') . 
$$
Towards this end, we note that the transition probability for the noisy process, denoted by $\{A_n\}$, can be written as (when $A_{n+1} \neq A_n$):
$$
P_{a,a'}(n)= P (A_{n+1} = a | A_{n}=a') =  \frac{g(a, a')}{g(a)}\Bigg( \big(1 - \epsilon_n\big) \exp\Bigg\{ \frac{- (\mu_{a'}-\mu_{a}  + \lambda(n))^+}{T_n} \Bigg\} +\epsilon_n\Bigg), 
$$
where $\lambda(n) = \{( \tilde{\mu}_{a'}(n) -\mu_{a'})  -(\tilde{\mu}_{a}(n) -\mu_{a}) \} $ is sub-Gaussian with variance $\sigma^2(n) := \text{Var} (\lambda(n))$. It is quite straightforward to see that  $\sigma^2(n)  \leq \frac{2\sigma^2}{\min\{T_{a}(n),T_{a'}(n)\}}$. So, we have 
\begin{equation}\label{subgaus}
P(|\lambda(n)| \geq x ) \leq  2 \exp \Bigg( -\frac{ x^2 \min\{T_{a}(n),T_{a'}(n)\}}{ 4 \sigma^2 }\Bigg).
\end{equation}
We consider the following equivalent definition for $P_{a,a'} $ (see eq. (4), \cite{gelfand}): 
\begin{align*}
P_{a,a'}(n) &= \EX_{\lambda(n)} [P \big(A_{n+1} = a_{n+1} | A_{n}=a_n,\lambda(n)].
\end{align*}
Then, we have
\begin{align*}
\xi_{n}  & = \Ap (\bar{A}_{n+1} = a | \bar{A}_{n}= a') - P (A_{n+1} = a | A_{n}=a').  \\
&\leq   \exp\Bigg( \frac{ (\mu_{a'}-\mu_{a}  + \lambda(n))^+ -(\mu_{a'}-\mu_{a} )^+ }{T_n} \Bigg)  -1 +\epsilon_n \\
 &\leq  \underbrace{ \EX_{\lambda(n)} e^{\frac{ |\lambda(n)|}{T_n} }    - 1 }_{\xi'_n} +\epsilon_n
\end{align*}
In obtaining the last inequality, we  have used the easily verifiable relation: $(a+b)^+ - a^+  \leq |b|$ for any $a,b \in \mathbb{R}$. Setting $b = \lambda(n)$  and $a= \mu_{a_{n+1}} - \mu_{a_n} $ in this relation gives the required inequality. Denote  $\beta_n:= \frac{1}{T_n}$, so that
\begin{align*}
\xi'_{n}  & \leq   \EX_{\lambda(n)}  e^{  \beta_n \big| \lambda(n ) \big|}  - 1 ,
\end{align*}
Using the relation $\EX \,X = \int_{x \geq 0} \mathbb{P} (X \geq x ) dx$ for any non-negative random variable $X$, we have 
\begin{align*}
\xi'_{n}  & \leq   \int_0^1  \mathbb{P}\Big(  \beta_n| \lambda_n  |> \log x    \Big)\,dx  + \int_1^\infty  \mathbb{P}\Big(\beta_n   |\lambda_n|   > \log x \Big) dx  -1 \\
&=  \int_1^\infty  \mathbb{P}\Big( \beta_n |\lambda_n|   > \log x \Big) dx.
\end{align*}
Let $N_n := \min\{T_{a}(n),T_{a'}(n)\}$. Using the sub-Gaussian property of $\lambda_n$ (see eq. (\ref{subgaus})), we have
$$
\xi'_{n} \leq   \int_1^\infty e^{-\Big(  \frac{\log x \sqrt{N_n} }{2\sigma \beta_n}\Big)^2} dx.
$$ 
Using the transformation $s = \frac{\log x}{2\alpha(n)}$ with $\alpha(n) = \frac{\sigma \beta_n}{\sqrt{N_n}}$, we have
\begin{align*}
	\xi'_{n}  &\leq    2 \alpha(n) \int_0^\infty e^{2\alpha(n)s} e^{- s^2} ds.\\
		&\leq 2 \alpha(n) e^{\alpha(n)^2}\int_0^\infty e^{-\alpha(n)^2 +2\alpha(n)s - s^2}ds.\\
		&= 2 \alpha(n) e^{\alpha(n)^2}\int_0^\infty e^{-(s - \alpha(n))^2}ds.\\
\end{align*}
To finish, we note that,
$$
 \int_{-\infty}^\infty \exp(-u^2) \,du \leq \sqrt{\pi}.
$$
Thus we have :
\begin{align*}
\xi'_{n}  &\leq   2\sqrt{\pi}\alpha(n) e^{\alpha(n)^2}\\
\implies \xi_{n}  &\leq   2\sqrt{\pi}\alpha(n) e^{\alpha(n)^2} +\epsilon_n
\end{align*} 
We state this result in the following lemma:
\begin{lem}\label{lem-bound}
For any $a,\,a'\in \A$, we have
$$
  \bar{P}_{a,a'}(n) -P_{a,a'}(n)\leq 2 \sqrt{\pi}\alpha(n) e^{\alpha(n)^2} + \epsilon_n,
$$ 
where $\alpha(n):=  \frac{\sigma \beta_n}{\sqrt{N_n}}$ with $\sigma$ being the variance of the noise, $\beta_n := \frac{1}{T_n}$ the inverse temperature and $N_n := \min\{T_{a}(n),T_{a'}(n)\}$. 
\end{lem}


\subsection{Construction of the Markov Process}

To establish the regret bound, it is convenient to switch to a continuous time analysis. This will allow us to use the concept of Dirichlet form which is central to the main arguments of the proofs. Before doing so, it is instructive to go through the standard construction of the Markov process since handling time inhomogeneity requires some care. We follow the construction of \cite{strockbook} here. Let $\{X(t)\,:\,t\in \mathbb{R}^+\}$ be the required Markov process having the transition mechanism $P(s,t)$ corresponding to the SA chain.  Let $\nu(0)$ be the initial distribution of the chain. So, we want to define a process $\{X(t)\,:\,t\in \mathbb{R}^+\}$ which satisfies:
\begin{equation}\label{1}
\Pm(X(0) = i ) = \nu_i(0) \,\,\,\,\text{}\,\,\,\,\, P(s,t)_{X(s)j} :=   \Pm(X(t)=j \,| \,X(\sigma),\,\sigma\in[0,s] ).
\end{equation}
A naive construction of $\{X(t)\}$ with the above transition mechanism would potentially mean dealing with an uncountable number of random variables  for each $(t,i) \in [0,\infty) \times \A $. To avoid this, we let $S(t,i,j) :=  \frac{1}{g(i)} $ $\sum_{m=1}^j g(i,m) e^{ - \beta(t) (\tilde{u}_m(t) -\tilde{u}_i(t)   )^+} $,where $\tilde{u}_i(t)$ is defined below. Define:
\begin{equation}
\Psi(t,i,u )= 
\begin{cases}
j  \,\,\,\,  \text{if},\,\,\,\,\,\,\,\,\, \frac{S(t,i,j-1)}{S(t,i,k)}\leq u < \frac{S(t,i,j)}{S(t,i,k)} \\
i  \,\,\,\, \text{if},\,\,\,\,\,\,\,\,\,\, u >1.
\end{cases}
\end{equation}
Also, we determine the function $\mathcal{T}$ by:
$$
\int_s^{s + \mathcal{T} (s,i,\xi)} S(\tau,i,k)^{-1}\,d\tau = \xi.
$$
Having defined the above functions, the rest of the construction is the same as that of a standard cadlag jump process: Let $X(0)$ be an $\A$-valued random variable with distribution $\nu(0)$ and $\{E_n : n > 1\}$ be a sequence of i.i.d.\ mean $1$ exponential random variables. Furthermore, let $\{U_n : n > 1\}$ be a sequence of i.i.d.\ random variables uniformly distributed on $[0,1)$ and also independent of the sigma algebra $ \sigma( X(0) \cup \{E_n : n > 1\})$. Finally, set $J_0 = 0$ and $X(0) = X_0$, and, when $n > 1$, we inductively define the following
$$
J_n - J_{n-1} = \mathcal{T} (J_{n-1},X(J_{n-1}),  E_n), \,\,\,\, X(J_n) =  \Psi(J_n , X(J_{n-1}),U_n),
$$
$$
X(t) := X(J_{n-1}), \,\, \text{for} \, \, J_{n-1} \leq t<J_n,
$$
$$
\tilde{u}_{i}(t) := \frac{1}{T_i(t)} \sum_{m =1}^n \mathbb{I} \Big\{X(t') =i,\,\, t'\in[J_{m-1},J_{m}) \Big\} \tilde{\mu}_i \,\,\text{for} \, \, J_{n-1} \leq t<J_n,\,\,\,i\in\A,
$$
where $\EX \, \tilde{\mu}_i = \mu_i$  and $T_i(t) :=  \sum_{m =1}^n \mathbb{I} \{X(t') =i,\,\, t'\in[J_{m-1},J_{m}) \} $ is the number of visits to arm $i$ till time $t \in [J_{n-1},J_n)$. A routine argument can show that for the above process, (\ref{1}) holds. The $Q$-matrix for the SA process with a continuous time temperature parameter $T_t : \mathbb{R} \to \mathbb{R}$ can be defined as:
\begin{equation}\label{q-matrixdef}
Q_{ij}(t) :=
\begin{cases}
\frac{g(i,j)}{g(i)}\Big( \big(1-\epsilon(t)\big)\exp\Big( - \beta(t)(\tilde{\mu}_j(t) - \tilde{\mu}_i(t))^+\Big) + \epsilon(t) \Big)\,\, \text{ for }\,\, j \neq i\\
-\sum_{i \neq j} Q_{ij}(t)
\end{cases}
\end{equation} 
We note that our $Q$-matrix is time dependent which potentially poses a problem towards deploying the standard Markov machinery (more precisely, the Kolmogorov forward equation). We therefore show that the forward equation holds almost everywhere because $Q(t)$ is continuous almost everywhere. We will prove the latter in what follows, i.e., we prove
\begin{equation}\label{q-matrix}
\frac{d}{dt} P(s,t) = P(s,t) Q(t) \,\,\text{almost surely on} (s,\infty) \,\,\,\text{with }P(s,s)=I.
\end{equation}
To prove (\ref{q-matrix}), we use an approximation argument. Accordingly, let
$$
Q^N(t) := Q([t]_N) \,\,\,\text{ where }[t]_N = \frac{J_n}{N} \text{ for } t\in \Bigg[\frac{J_n}{N},\frac{J_{n+1}}{N}\Bigg)\,\,\,\text{for }t>s
$$
It is obvious, the solution to (\ref{q-matrix}) with $Q(t)$ replaced with $Q^N(t)$, is given by
$$
P^N(s,t) = P^N (s,s\vee [t]_N )e^{(t-s)Q([t]_N)}.
$$
We note that :
\begin{multline*}
\| P^N(s,t)  -P^M(s,t) \| \leq \int_s^t \| Q([\tau]_N)-Q([\tau]_M)  \| \,\|P^N(s,\tau) \| d\tau +  \int_s^t \|Q([\tau]_M) \|\, \|  P^N(s,\tau) -P^M(s,\tau)  \|  d\tau .
\end{multline*}
Let $G$ denote the matrix with entries $g_{ij}:= \frac{g(i,j)}{\sum_{j\in \N(A)} g(i,j)}$ and $g_{ii}=0$. 
So, the above inequality gives
\begin{equation*}
\| P^N(s,t)  -P^M(s,t) \| \leq \int_s^t \| Q([\tau]_N)-Q([\tau]_M)  \| d\tau + 
 \|G\|\int_s^t \|  P^N(s,\tau) -P^M(s,\tau)  \|  d\tau. 
\end{equation*}
An application of Gronwall's lemma gives
\begin{equation*}
\sup_{0\leq s\leq t\leq T}\| P^N(s,t)  -P^M(s,t) \|  =\mathcal{O} \Big( e^{GT}  \int_0^T \| Q([\tau]_N)-Q([\tau]_M)  \|  d\tau \Big).
\end{equation*}
We note that as $N,M \to \infty$, then $|\beta([\tau]_N) - \beta([\tau]_M)|\to 0$ implying that $ \| Q([\tau]_N)-Q([\tau]_M)  \|   \to 0$. So the above equation guarantees that  $P^N(s,t)$ converges to  some $(s,t) \to P(s,t)$ on finite intervals. In particular, this implies that
\begin{equation}\label{diff-2}
P(s,t) = I + \int_s^t P(s,\tau)Q(\tau) d\tau
\end{equation}
which is the integrated form of (\ref{q-matrix}). It can then be seen by routine arguments that
\begin{equation*}\label{diff-1}
\dot{\nu}(t) = \nu(t)Q(t)\,\,\text{ for }\,\,t\in [s,\infty) \iff \nu(t) = \nu(s)P(s,t) \,\,\text{ for }\,\,t\in [s,\infty). 
\end{equation*}
for $P(s,t)$ satisfying (\ref{diff-2}). This completes the construction of the process $\{X(t)\}$. 
\subsection{Regret Bound }
To begin establishing the regret bound, we briefly recall the details of some technical concepts related to Markov processes. Let $f$ denote a column vector determined by any function $f:\A \to \mathbb{R}^{k}$. Also, let $\pi(t)$ denote the stationary distribution of $P(t)$, determined by the continuous cooling schedule $T(t)=\frac{\gamma}{\log(t+t_0)}$, where $t\in \mathbb{R}$. We will use the following notation throughout this section:
$$
\|f \|_{\pi}  = \sqrt{\langle|f|^2 \rangle_{\pi}} \,\,\text{ where }\,\, \langle g \rangle_{\pi} := \sum_{i \in \A} g_i \pi_i. 
$$
Since $\pi_i(t) >0$ for each $i \in \A$, we note that $\|\cdot \|_{\pi} $ is a complete norm. The inner product between any two functions $f,g : \A \to \mathbb{R}^{k}$ is denoted by  $\langle f,g\rangle_{\pi}:= \sum_{i\in \A} \pi_i f_ig_i$\\

\textbf{Definition 5} (\textit{Dirichlet Forms and Poincar\'{e} inequality}) Let $L^2(\pi)$ denote the space of $f$ for which $\|f\|_{\pi} <\infty $. This space is actually a Hilbert space for which the transition matrix $P$ acts as a self adjoint contraction assuming the detailed balance equations are satisfied. The variance of $f$ is then defined to be:
$$
\text{Var}_{\pi}(f) := \|f- \langle f \rangle_{\pi}\|_{\pi}^2
$$
The Dirichlet form is defined as:
\begin{equation}\label{d-form}
\E_t(f,f) := \frac{1}{2} \sum_{\substack{i\in \A\\ j \neq i}} \pi_i Q_{ij}(t) (f_j-f_i)^2.
\end{equation}
Using the Dirichlet form, the Poincar\'{e}  constant (assuming fixed $t$) can be defined as
$$
\xi_+ = \inf \{\E_t(f,f)\,:\, f \in L^2(\pi) \text{ and } \text{Var}_{\pi} (f) =1 \}
$$
which gives us the well known Poincar\'{e}  inequality:
$$
\xi_+ \text{Var}_{\pi} (f) \leq \E_t(f,f),\,\,f\in L^2(\pi).
$$
For our purposes, we need the definition of $\xi_+$ when $\pi:=\pi(t)$. We denote this by
$$
\lambda(t) := \inf\{ \E_t(f,f)  :\,\,\text{Var}_{\pi(t)}(f) =1\}.
$$

The upper (and lower) bound on  $\lambda (t)$ has been provided in (Theorem 5.4.11, \cite{strockbook}): 
\begin{equation}\label{dirtc}
 \pi_{-}e^{-\beta(t)\gamma^*}\leq \lambda(t) \leq \pi_+e^{-\beta(t)\gamma^*},
\end{equation}
where $ \pi_{-},  \pi_+$ are bounded constants, $\gamma^*$ is the critical depth and $\beta(t):=\frac{1}{T_t}$ is the inverse temperature.
\[ \]
We use  the noise perturbation bound derived in the previous section to establish related upper bounds for the $Q$-matrix perturbation with respect to the noiseless process.  Accordingly, let
\begin{equation}\label{sigma-eq}
\sigma_{ij}(t) :=   \bar{Q}_{ij}(t) -Q_{ij}(t) 
\end{equation}
where $\bar{Q}(t) := Q$-matrix corresponding to the continuous time version of the undisturbed process $\{a_n\}$. Then, by Lemma \ref{lem-bound}, we have
$$
  \sigma_{ij}(t)    \leq 2 \sqrt{\pi}\alpha(t) e^{\alpha(t)^2} +\epsilon(t),
$$
where $\alpha(t) := \frac{\sigma \beta(t)}{\sqrt{ \min\{T_{i}(t),T_{j}(t)\}}}$, $\beta(t) = \frac{\log (t+1)}{\gamma}$ and $\epsilon(t) = \frac{1}{t^\frac{1-\epsilon}{d}}$. Assume without any loss of generality that $\frac{\log t}{t^{\frac{\epsilon}{4}}} \leq 1$. We have $\alpha(t) \leq \frac{\sigma \beta(t)}{\sqrt{\bar{T}(t)}}$, where $\bar{T}(t) := \min_{a\in \A }T_a(t)$. Set $\alpha(t) := c \beta(t)$, where $c:=\frac{\sigma }{\sqrt{\bar{T}(t)}}$. Suppose $T_a(t) := \mathcal{O} (t^\epsilon)$ for any $\epsilon>0$ and $a\in \mathcal{A}$.  Then,  we have
\begin{equation}\label{sigma-bd}
 \sigma_{ij}(t)    \leq   \frac{2 \sqrt{\pi} \sigma e^{ \frac{\sigma^2 (\log t)^2}{\gamma^2 t^\epsilon}}\log t}{\gamma t^{\frac{\epsilon}{2}} } + \epsilon(t) ,
\end{equation}
So, we have
\begin{align}\label{eps-bound}
 \int_{0}^t \sigma(\tau) d\tau &\leq \int_{0}^{t} \big(2 \sqrt{\pi}c\beta(\tau) e^{(c\beta(\tau))^2}  + \epsilon(\tau) \big) d\tau \nonumber \\
&\leq   \Bigg(\frac{2 \sqrt{\pi} \sigma e^{ \frac{\sigma^2}{\gamma^2}} \log t}{\gamma (1- \frac{\epsilon}{2}) t^{ \frac{\epsilon}{4}}} + \frac{1}{ (1-\frac{1-\epsilon}{d})t^{ \frac{1-\epsilon}{d} - \frac{\epsilon}{4}}} \Bigg)t^{1 -\frac{\epsilon}{4}} \nonumber\\ 
&:=Mt^{1 -\frac{\epsilon}{4}} .
\end{align}
where $M < \infty$ is a finite constant independent of $t$ if $\epsilon \in (0,\frac{4}{d+4} )$. We note that in proving the above bound we have assumed that $T_i(t) := \mathcal{O} (t^\epsilon)$. This can be easily guaranteed in expectation by the assumed conditions on $\epsilon(t)$. We have 
$$ \EX\, T_i (n) := \sum_{p=0}^{n-1} \EX \,\mathbb{I}\{a_{p+1} = i\}  = \sum_{p=0}^{n-1} P( a_{p+1} = i) $$
Since $\epsilon_n$ assigns a mass of at least $\frac{1}{k}$ to $i$ when $i \in \mathcal{N}(a_p)$ and $0$ otherwise, we have
$$
 \sum_{p=0}^{n-1}  P( a_{p+1} = i )
\geq   \frac{1}{k} \sum_{j \in \N(i)}  \sum_{p=1}^{n-1} \epsilon_p \mathbb{I}  \{ a_p = j\},
$$
$$\implies  \EX \,T_i (n) \geq  \frac{1}{k} \sum_{j \in \N(i)}  \sum_{p=1}^{n-1}  \epsilon_p  \mathbb{I}  \{ a_p = j\}
$$
By repeating the above calculation, we get (since $\epsilon_{p-1} > \epsilon_p$), 
$$\implies  \EX\, T_i (n) \geq  \frac{1}{k^2}\sum_{k \in \N(j)}  \sum_{j \in \N(i)}  \sum_{p=2}^n  \epsilon^2_{p} \mathbb{I}  \{ a_{p-1} = k\}
$$
Let $d = \max_{a \in \A} D_a$, where $D_a$ is defined as:
\begin{equation}\label{qt}
D_a = \min_{k} \Big\{ (a, \mathcal{N}(a)); \cup_{a_1 \in \N(a)} (a_1, \N(a_1)),\cdots, \cup_{a_k \in \N(a_{k-1})} (a_k, \N(a_{k}))  \Big\},\,\,\,\text{such that} \,\,\,A \subseteq D_a
\end{equation}
Using the above definition, it is obvious that
$$\implies  \EX\, T_i (n) \geq  \frac{1}{k^d}  \sum_{p=d}^n \epsilon^d_p
$$
which proves that $\EX \,T_i (n)   =  \mathcal{O} (n^\epsilon )$. On a side node, an application of Borel-Cantelli Lemma also shows $\sum_{p=0}^n \mathbb{I}\{a_{p+1} = i\} \to \infty $ as $n \to \infty$. Thus, from law of large numbers $\hat{\mu_i} \to \mu_i$.
\begin{rem}
If we take $\epsilon_n=0$, then one can instead execute a simple depth first search algorithm (which has a worst case complexity of $\mathcal{O}(k + e )$, with $e:=|\E|$ denoting the number of edges) to allocate a prespecified $T_i(t_0):=  =  \mathcal{O} (T_{\text{hor}}^\epsilon) $ budget to each arm. Another possible strategy is to just initialize the empirical mean to $-\infty$ and freeze the temperature for a certain pre-determined time duration. Since our graph is connected, we are guaranteed to visit all arms and allocate the preliminary budget to them. 
\end{rem}
\noindent For reference, we state the previous results in the following lemma:
\begin{lem}\label{newlem}
Suppose  $\epsilon(t) = \frac{1}{t^\epsilon}$ with $\epsilon \in \Big(0,\frac{4}{d+4} \Big)$, where $d$ is defined in (\ref{qt}) and $T_a(t)= \mathcal{O}(t^\epsilon)$ for all $a \in \A$. Then, we have 
\begin{align*}
 \int_{0}^t \sigma(\tau) d\tau &\leq  Mt^{1 -\frac{\epsilon}{4}},
\end{align*}
where,
$$
M \leq \Bigg(\frac{2 \sqrt{\pi} \sigma e^{ \frac{\sigma^2}{\gamma^2}} }{\gamma (1- \frac{\epsilon}{2}) } + \frac{1}{ (1-\frac{1-\epsilon}{d})} \Bigg)
$$
\end{lem}

Just as in the noiseless case, in order to establish a regret bound, we first bound the probability of the event $X(t) \in \A/\{a^*\}$. 

\begin{lem}\label{problem}
Let $\beta(t) = \frac{\log \big(1 + t\big)}{\gamma}$ and $\epsilon(t) = \frac{1}{t^{\epsilon}}$ where $\epsilon \in (0,4(d+4)^{-1})$. Then, for $\gamma >\frac{4\gamma^*}{\epsilon}$, we have for any $t > 0 $, 
$$
P\Big(X(t) \in \A/\{a^*\}\Big) \leq  \sqrt{ \frac{2g(k + 4)}{(1+ t)^{ \frac{\Delta_{\text{min}}}{\gamma}}}},
$$
where $g: = 2\sum_{i\in \A/a^*} \frac{g(a)}{g(a^*)}$, $\Delta_{\text{min}} : = \min_{i \in \A} (\mu_{a_i} - \mu_{a^*})$ and $k$ is the number of arms. 
\end{lem}

\begin{proof}
Let $\nu(t)=\nu(0)P(0,t)$ for any $t\geq 0$. We denote the Radon-Nikodym derivative of $\nu(t)$ w.r.t $\pi(t)$  by:
$$
f_i(t) = \frac{\nu_i(t)}{\pi_i(t)},\,\,\forall \,t\geq 0,\,i \in \A.
$$
To get an estimate on the probability of selecting a bad arm, we note that
\begin{align*}
P\Big(X(t) \in \A/\{a^*\} \Big) &=  \sum_{i \in \A / \{a^*\}}  \nu_i(t)  \\
& = \sum_{i \in \A / \{a^*\}} \pi_i(t) f_i(t)\\
& = \langle f(t),\mathbb{I}\{i \in \A/a^* \} \rangle_{\pi(t)}\\
&\leq \|f(t)\|_{\pi(t)} \sqrt{\sum_{ i \in \A/\{a^*\}} \pi_i(t)} .
\end{align*}
Suppose that $\|f(t)\|_{\pi(t)} \leq C$ for some constant $C>0$, then we would have
\begin{align*}
P(X(t) \in \A/\{a^*\}) &\leq \frac{Cg}{(1+t)^{ \frac{\Delta_{\text{min}}}{2\gamma}}},
\end{align*}
where we have used the continuous version of the previously established estimate of the decay rate of the Gibbs measure to the Dirac measure on $a^*$  (see (\ref{term2})). The rest of the proof bounds $\|f(t)\|_{\pi(t)}$ by providing an upper bound on the constant $C$. \\

Let $Z(t) : = \sum_{a \in \A} g(a) e^{-\beta(t)\mu_a }$ denote the partition function, so that $\pi_a(t) = g(a) e^{-\beta(t)\mu_a }/Z(t)$. Then, we have  by the easily verifiable relation $\dot{Z}(t) = -\dot{\beta}(t) Z(t) \langle \mu\rangle_{\pi(t)} $, that
\begin{align}\label{est-1}
	\frac{d}{dt} \|f(t)\|^2_{\pi(t)} &=  	\frac{d}{dt} \Bigg( \frac{1}{Z(t)} \sum_{i \in \A} (f_i(t))^2 e^{-\beta(t) \mu_i} \Bigg) \nonumber \\
	&=  \frac{2}{Z(t)} \sum_{i \in \A} f_i(t) \dot{f}_i(t) e^{-\beta(t) \mu_i}  -  \frac{\dot{\beta}(t) }{Z(t)}\sum_{i \in \A} \mu_i f_i^2(t)e^{-\beta(t) \mu_i} \nonumber\\
	& \,\,\,\,\,\,\,\,\,\,\,\,\,\,\,   -   \frac{1}{Z^2(t)}\sum_{i \in \A} \dot{Z}(t)  f_i^2(t) e^{-\beta(t) \mu_i}  \nonumber \\
	&= 2\langle f(t),\dot{f}(t) \rangle_{\pi(t)} -\dot{\beta}(t)  \langle\mu -\langle \mu \rangle_{\pi(t)},f^2(t) \rangle_{\pi(t)}.
\end{align}
We can establish a different  upper bound on $\frac{d}{dt} \|f(t)\|^2_{\pi(t)} $ by noting that 
\begin{align*}
\|f(t)\|^2_{\pi(t)} &= \langle f(t)\rangle_{\nu(t)}\\ 
&= \sum_{j \in \A} \bigg( \sum_{i \in \A}\nu_i(0)P_{ij}(0,t)\bigg) f_j(t) \\
&= \sum_{i \in \A} \nu_i(0) \bigg( \sum_{j \in \A} P_{ij}(0,t)f_j(t)\bigg)  \\
&=  \langle P(0,t)f(t)\rangle_{\nu(0)}.
\end{align*}
Differentiating the above equality:
\begin{align}\label{est-2}
\frac{d}{dt} \|f\|_{\pi(t)}^2 &= \langle P(0,t)Q(t) f(t)  \rangle_{\nu(0)} + \langle P(0,t) \dot{f}(t)\rangle_{\nu(0)} \nonumber \\
&= \Bigg( \sum_{i \in \A}  \Big(\sum_{p \in \A}\nu_p(0)P_{pi}(0,t)\Big) \Big(\sum_{j \in \A} q_{ij}(t)f_j(t) \Big)\Bigg) + \sum_{i \in \A} \Big( \sum_{p \in \A}\nu_p(0)P_{pi}(0,t)\Big) \dot{f}_i(t) \nonumber \\
&= \sum_{i \in \A}  f_i(t) \pi_i(t)\Big( \sum_{j \in \A} q_{ij}(t)f_j(t) \Big) + \sum_{i \in \A} f_i(t) \pi_i(t) \dot{f}_i(t) \nonumber \\
&= \sum_{i \in \A}   \pi_i(t)\Big( -\sum_{j \in \A/\{i\}} q_{ij}(t)f^2_i(t) +\sum_{j \in \A/\{i\}} q_{ij}(t)f_i(t)\ f_j(t) \Big) +  \langle f(t),\dot{f}(t) \rangle_{\pi(t)}  \nonumber \\
&=  -\mathcal{E}_t(f(t),f(t))   + \langle f(t),\dot{f}(t) \rangle_{\pi(t)} .
\end{align} 
where we have used (\ref{q-matrix}) in the first term of the first equation, (\ref{q-matrixdef}) in the fourth equation and the fact that $\pi_i(t)q_{ij}(t) = \pi_j(t)q_{ji}(t) $ along with the definition of (\ref{d-form}) in deriving the last equation. 
Define $ \bar{\mathcal{E}}_t(f(t),f(t))$ as
$$ \bar{\mathcal{E}}_t(f(t),f(t)):=  \frac{1}{2}  \sum_{\substack{i\in \A\\ j \neq i}}  \pi_i(t) \bar{Q}_{ij}(t)   \big( f(j)  - f(i)\big)^2.$$ 
Combining estimates (\ref{est-1}) and (\ref{est-2}) to remove the $\langle f(t),\dot{f}(t) \rangle $ term, we have:
\begin{align}\label{jet}
	\frac{d}{dt}\|f(t) \|^2_{\pi(t)} & = -2\mathcal{E}_t(f(t),f(t)) + \dot{\beta}(t)\langle \mu - \langle \mu \rangle_{\pi(t)}    , f^2(t)\rangle_{\pi(t)}  \nonumber \\
&= -2\bar{\mathcal{E}}_t(f(t),f(t)) + \dot{\beta}(t) \langle \mu - \langle\mu \rangle_{\pi(t)}  , f^2(t) \rangle_{\pi(t)} +  2\bar{\mathcal{E}}_t(f(t),f(t)) - 2\mathcal{E}_t(f(t),f(t)) , \nonumber \\
&\leq -2\lambda(t) \|f(t)\|^2_{\pi(t)} +2\lambda(t) +   \dot{\beta}(t)\mu_{\text{max}}\|f^2(t) \|_{\pi(t)} +  2\bar{\mathcal{E}}_t(f(t),f(t)) - 2\mathcal{E}_t(f(t),f(t)) \nonumber \\
&\leq -\lambda(t) \|f(t)\|^2_{\pi(t)} +2\lambda(t) +  2\bar{\mathcal{E}}_t(f(t),f(t)) - 2\mathcal{E}_t(f(t),f(t)) ,
\end{align}
where we have used the facts: (i) $\bar{\mathcal{E}}_t(f(t),f(t)) \geq \text{Var}_{\pi(t)}(f(t))\lambda(t)=(\|f\|^2_{\pi(t)} -1) \lambda(t) $ in the first term in the third inequality, and, (ii) $ \mu_{\text{max}}  \dot{\beta}(t) \leq  \lambda_{-}e^{-\beta(t)\gamma^*} \leq \lambda(t)$ by definition of $\beta(t)$ in the last inequality.

We also have (see eq. (\ref{sigma-eq}))
\begin{align*}
	 \bar{\mathcal{E}}_t(f(t),f(t)) - \mathcal{E}_t(f(t),f(t)) & = \frac{1}{2} \sum_{\substack{i\in \A\\ j \neq i}}  \pi_i(t) \sigma_{ij}(t) \big(f(j)  - f(i)\big)^2,
\end{align*}
so that
 \begin{align*}
	 \bar{\mathcal{E}}_t(f(t),f(t)) - \mathcal{E}_t(f(t),f(t))  &  \leq \frac{1}{2} \max_{i,j  ; i \neq j}  \sigma_{ij}(t)  \sum_{\substack{i\in \A\\ j \neq i}}  \pi_i(t)  \big(f(j)  - f(i)\big)^2.
\end{align*}
Let $\sigma(t) = \max_{i,j;  i \neq j}  \sigma_{ij}(t) $. The above inequality gives
 \begin{align*}
 \bar{\mathcal{E}}_t(f(t),f(t)) - \mathcal{E}_t(f(t),f(t))  & \leq k\sigma(t) \| f(t)\|^2_{\pi(t)}.
\end{align*}
Plugging this back in (\ref{jet}) and w.l.o.g. absorbing the factor of $k$ into the $\sigma(t)$ term , we have 
\begin{equation*}
\frac{d}{dt}\|f(t) \|^2_{\pi(t)} \leq -\big(\lambda(t)  -\sigma(t)\big) \|f(t) \|^2_{\pi(t)} + 2\lambda(t)
\end{equation*}
\begin{equation*}
\implies	\frac{d}{dt}\|f(t) \|^2_{\pi(t)}  + \big(\lambda(t)  -\sigma(t)\big) \|f(t) \|^2_{\pi(t)} \leq 2\lambda(t).
\end{equation*}
Set $\rho(t) =  \int_{0}^t \lambda(t) dt -   \int_{0}^t \sigma(t)dt $ and multiply both sides of the above inequality by $e^{\rho(t)}$ to get 	
\begin{equation*}
\frac{d}{dt} e^{\rho(t)} \|f(t) \|^2_{\pi(t)}  \leq 2\lambda(t) e^{\rho(t)}.
\end{equation*}
Thus,
\begin{equation}\label{finaleq}
  \|f(t) \|^2_{\pi(t)}   \leq  \frac{1}{e^{\rho(t)}}\|f(0)\|^2_{\pi(0)} + \frac{2}{e^{\rho(t)}} \int_{0}^t  \lambda(\tau) e^{\rho(\tau)} d\tau.
\end{equation}
We first establish conditions on $\gamma$ for  $\int_{0}^t \lambda(t) dt >  \int_{0}^t \sigma(t)dt$ to hold. We recall that $\lambda(t) \geq \pi_- e^{-\frac{\beta(t)}{\gamma^*}} = \pi_-(1+t)^{-{\frac{\gamma^*}{\gamma}}} $. Also, as discussed previously,  $\int_{0}^t \sigma(t)dt =\mathcal{O}( t^{1 -\frac{\epsilon}{2}} )$ (see Lemma \ref{newlem}). So, for the required condition to hold, we must have:
$$
 \int_{0}^t \lambda(t) dt >  \int_{0}^t \sigma(t)dt \iff \frac{ \pi_-}{1-\frac{\gamma^*}{\gamma} } (1+t)^{1-\frac{\gamma^*}{\gamma}} >   \int_{0}^t \sigma(t)dt 
$$
or equivalenty,
$$
 \frac{ \pi_-}{1-\frac{\gamma^*}{\gamma} } (1+t)^{1-\frac{\gamma^*}{\gamma}} > Mt^{1-\frac{\epsilon}{4}}
$$
The above condition will be satisfied for $\gamma>\frac{ 4\gamma^*}{\epsilon}$ for some $t >t_0$ (assume $t_0=0$ for simplicity). We next consider the integral term $\int_{0}^t  \lambda(\tau) e^{\rho(\tau)} d\tau$:
\begin{align*}
\int_{0}^t  \lambda(\tau) e^{\rho(\tau)} d\tau  & = \int_{0}^t  \frac{ \lambda(\tau) - \sigma(\tau) }{1 - \frac{ \sigma(\tau) }{\lambda(\tau)}} e^{ \int_{0}^\tau \lambda(t') dt' - \int_{0}^\tau \sigma(t')dt'}    d \tau\\
&= \int_{0}^t  \frac{ 1 }{1 - \frac{ \sigma(\tau) }{\lambda(\tau)}} \Bigg( \frac{d}{d\tau}e^{ \rho(\tau)} \Bigg)   d \tau.
\end{align*}
We also have
\begin{align*}
 \frac{\sigma(\tau)}{\lambda(\tau)} &\leq \frac{M }{\pi_- } \tau^{ \frac{\gamma^*}{\gamma}- \frac{\epsilon}{4}} \\
&\leq \Big( \frac{M }{\pi_-}  \Big) \frac{1}{\tau^{\frac{\epsilon}{4} - \frac{\gamma^*}{\gamma} } } \\
&\leq \frac{1}{2} ,
\end{align*}
if $\tau \geq \Big( \frac{2M }{\pi_- }  \Big)^{\frac{1}{\frac{\epsilon}{4} - \frac{\gamma^*}{\gamma}}} := \tau_0$. So, we have
\begin{align*}
\int_{0}^t  \lambda(\tau) e^{\rho(\tau)} d\tau  &\leq   2 \Big( e^{\rho(t)}    - e^{\rho(\tau_0)}\Big).
\end{align*}
Plugging this back in (\ref{finaleq}), we get
\begin{equation*}
  \|f(t) \|^2_{\pi(t)}   \leq \frac{ \|f(\tau_0)\|^2_{\pi(\tau_0)}}{e^{\rho(t)}} + 4\Big(1 - \frac{e^{\rho(\tau_0)}}{e^{\rho(t)}} \Big) \leq 4+ \|f(\tau_0)\|^2_{\pi(\tau_0)}\leq  4+ k.
\end{equation*}
This bounds $ \|f(t) \|^2_{\pi(t)}$ and hence concludes the proof.
\end{proof}
One can deduce from the previous result, that the convergence rate is upper bounded by $t^{- \frac{\Delta_{\text{min}}}{2\gamma}}$.  Depending upon the bound on $\gamma $ which in turn depends on the energy landscape (through $\gamma^*$) and the parameter $\epsilon$, this can be arbitrarily bad compared to the fully connected setting for which $\gamma^*,d=0$. 

To establish the regret, we use can proceed the same way as in Theorem 2.  W first recover the discrete time process $\{a_n\}_{n\geq 1}$ from $\{X(t)\}_{t>0}$ by using the following identification: $a_n = X(t), $ for $t \in[J_{n},J_{n+1})$. Then, we have, 
\begin{thm}
Assume the conditions of lemma \ref{newlem} and lemma \ref{problem} are fulfilled. Then, the regret of Algorithm 1 is bounded by
$$
R_n \leq  \Big( \frac{2M }{\pi_- }  \Big)^{\frac{1}{\frac{\epsilon}{4} - \frac{\gamma^*}{\gamma}}} +    \frac{\sqrt{ g(4 + k)}}{\big( 1-\frac{\Delta_{\text{min}}}{2\gamma}\big)}   (n+1)^{1- \frac{\Delta_{\text{min}}}{2\gamma}}
$$
with $\gamma> \frac{4\gamma^*}{  \epsilon}$.
\end{thm}


\begin{rem}
The monotonically decreasing cooling schedule of simulated annealing has been a point of debate ever since it was proposed. The case of noisy observations, as one can guess, inherits these problems. It is a common practice to employ a constant time schedule in bandit algorithms (see e.g. the detailed experiments performed in \cite{krup}), even though the theoretical guarantees may not be exact. We can also suggest the following alternative here: Let $T_{e_{ij}}(t)$ denote the number of visits to edge $e_{ij} \in \A\times\A$ determined by $i,\,j$, then we defined the cooling schedule as:
$$
\beta(t) := \frac{\log T_{e_{a_n,a'}} (t)}{\gamma},
$$
where $a' \in \N(a_n)$ is the candidate arm uniformly selected from the neighbourhood of the current arm $a_n$. This makes the cooling schedule depend on the state of the algorithm. By a limiting pigeon-hole argument for a finite action set (which implies that at least one edge in $\G$ will be visited infinitely often), one can see that $\limsup_t \beta(t) \to \infty$. Although, the previous results may not hold with the same guarantees since the time dependence in the upper bound will be through $T_{e_{a_n,a'}}(t)$ instead of $t$, we conjecture that this time step may lead to a more adaptive approach to exploration.

\end{rem}
$$
$$

\section{Stochastic Multi-Armed Bandit}
For completion, in this section, we consider the standard stochastic MAB problem which can be considered the special case of the previous section. To be more precise, we do not assume a graphical structure on $\A$, so that  one may think of the graph as being fully connected. It is, however,  inadvisable to carry out a straightforward naive implementation of the simulated annealing algorithm . Instead, we propose a slight modification,  which takes into account the additional structure of a fully connected setting. 

\begin{algorithm}[htb]
\textbf{Input:} Arm set $\A = [k]$;  Time horizon $n$; Temperature Parameter $\gamma$.
\\
\textbf{Initialization:} (i) Initialize the number of visits to any arm till time $p$, denoted by $T_a(p) $, to $0$. (ii) Initialize the empirical mean estimate $\hat{\mu}_{i,T_a(0)} =-\infty$, for all $i \in [k]$.
\\
\For{$p=0,\cdots,n$ } { 
            \textbf{(a)}  If an arm has not been played for $ \Big\lceil \frac{16 \log n }{ \gamma^2}\Big\rceil$ times, play it or else select arm $a_{\text{min}} $ according to:
  		    $$
  		    a_{\text{min}} = \argmin_{i \in [k]} \hat{\mu}_{i,T_i(p)}.
  		    $$
          \textbf{(b)} Uniformly randomly select an arm $a' \in \A/\{a_{\text{min}}\}$. Accept the transition according to
          $$
  			 P_{a_{\text{min}},  a' } = \exp \Big(- \frac{ \big(\hat{\mu}_{a',T_{a' }(p) } -\hat{\mu}_{a_{\text{min}},T_{a_{\text{min}} }(p) } \big)^+}{T_p} \Big).
 		$$
		 \textbf{(c)} Pull the currently selected arm $a_{p+1}$ to get the sample $\mu_{a_{p+1}}'$, and update the mean according to 
		 $$
		 \hat{\mu}_{a_{p+1},T_{a_{p+1} }(p) +1} = \Big(1 - \frac{1}{T_{a_{p+1}}(p)+1}  \Big) \hat{\mu}_{a_{p+1},T_{a_{p+1} }(p) }+ \frac{\mu_{a_{p+1}}'}{T_{a_{p}+1}(p)+1} .
		 $$
		 \textbf{(d)} Update the cooling schedule by setting $T_{p+1} =  \frac{\gamma}{\log (p + 1)}$.\\
		 
}
 
 \textbf{Output}: Resulting Policy: $\nu(n)$ \\

\caption{ SA Bandits}
\end{algorithm}
 We devote this current section for formally proposing and analysing this algorithm. The pseudo code is provided in Algorithm 2. A key difference from traditional SA lies in how we select a root node from which the ensuing exploratory move may be performed. One can note that the exploration process is no longer Markovian in nature. However, selecting the $\argmin_{i} \hat{\mu_i}(\cdot) $ in step(a) is not necessary as discussed in the following remark. 
 
 \begin{rem}
We can retain the original structure of the SA algorithm, i.e. use the previously selected state in place of $\argmin_{i} \hat{\mu_i}(\cdot) $ to perform the exploratory move. The regret will be still logarithmic but the $(\Delta_{i}-\gamma)^2$ term in the denominator of the second term in (\ref{reg-ba}) in Theorem 5 will change to $ \Big( \min( \mu_i - \mu_{i-1}, \mu_{i+1} -\mu_i ) -\gamma \Big)^2 $, if the initial number of pulls for each arm is $\lceil \frac{16 \log n }{ (\min( \mu_i - \mu_{i-1}, \mu_{i+1} -\mu_i ) -\gamma)^2} \rceil $. 
 \end{rem}
 
 \begin{rem}
The requirement of the initial pulls of $ \Big\lceil \frac{16 \log n }{ \gamma^2}\Big\rceil$ is again also not necessary from an empirical standpoint. One can also use a different strategy where once an arm is uniformly selected in Step (b), the probability transition mechanism is overridden and the arm is pulled with probability one if $T_a(t) \leq \Big\lceil \frac{16 \log n }{ \gamma^2}\Big\rceil$. This introduces an additional (constant) term in the regret which does  not arise for Algorithm 2. 
\end{rem}
The next theorem gives the upper bound on the regret of Algorithm 2. This turns out to be logarithmic, thereby  establishing the efficacy  of Algorithm 2 in solving the MAB problem. Without loss of generality, we assume that the first arm is optimal, so that $\mu_{a^*} = \mu_1$, and $\mu_1\leq \mu_2\leq\cdots\leq\mu_k$ and $\Delta_i := \mu_i -\mu_1$.
\begin{thm}
If $\gamma \in (0,\Delta_2)$, then the regret of Algorithm 2 on any bandit $P \in \mathcal{E}^{k}_{\text{SG}}(1) $ environment, is bounded by
\begin{equation}\label{reg-ba}
R_n \leq  \Big(2k+\frac{ \log n}{k-1}\Big) \sum_{i=1}^{k} \Delta_i + \sum_{i=1}^k \frac{16 \Delta_i   }{ (\Delta_i - \gamma)^2}   \log n,
\end{equation}
where $n$ is the number of rounds and $k$ is the number of arms .
\end{thm}

\begin{proof}
 We recall the equation,
$$
R_n = \sum_{i=1}^k \Delta_i \EX[T_i(n)].
$$
We proceed by considering a suitably defined event $G$ and subsequently bounding $\EX[T_i(n)]$ for each sub-optimal arm $i \in \A/\{a_1\}$ on $G$ (and its complement). Accordingly, the event $G$ is defined to be:
\begin{multline*}
    G := \Big\{ \mu_1 \leq \min_{s \in [n]} \{\Mu_{1,s} + \epsilon_{1,s} \} \Big\} \cap  \Big\{ \max_{s \in [n]}  \{\Mu_{1,s} - \epsilon_{1,s} \} \leq \mu_1\leq \mu_2 \leq  \min_{s \in [n]} \{\Mu_{2,s} + \epsilon_{2,s}  \} \Big\}...\\
    .....\cap \Big\{ \max_{s \in [n]}  \{\Mu_{k-1,s} - \epsilon_{k-1,s} \} \leq \mu_{k-1}\leq \mu_{k}  \leq  \min_{s \in [n]} \{ \Mu_{k,s} + \epsilon_{k,s} \}\Big\}... \\
     ....\cap \{\min_{s \in [n]} \{\Mu_{k,s} - \epsilon_{k,s}\} \leq \mu_k \Big\},
\end{multline*}
where,
$$
\epsilon_{i,s} = \sqrt{\frac{2 \log\frac{1}{\delta}}{s}}.
$$
We specify the value of $\delta$ below. We essentially proceed by demonstrating two facts which will allow us to bound $\EX[T_i(n)]$ :\\

1. The complement event $G^c$ occurs with sufficiently low probability if we set $\delta = \frac{1}{n^2}$. \\

2. On the event $G$, if each arm has been been played at least $u_i  = \frac{16 \log n}{(\Delta_i-\gamma)^2}$ times, the probability of transitioning to a sub-optimal arm diminishes as $\mathcal{O}(\frac{1}{t})$ for round $t$.

$$
$$
We have:
\begin{equation}\label{mainbanditregret}
    \EX[T_i(n)] = \EX[\mathbb{I}\{G\}T_{i} (n)]+ \EX[\mathbb{I}\{G^c\}T_{i} (n)]. 
\end{equation}
We first consider fact 1. Let us define an event $G^c_i$ as
$$
G^c_i =  \Big\{ \mu_i \geq \min_{s \in [n]} \{\Mu_{i,s} + \epsilon_{i,s} \}  \Big\} \cup  \Big\{  \mu_i \leq \max_{s \in [n]}  \{ \Mu_{i,s} - \epsilon_{i,s} \}  \Big\}.
$$
We remark in passing that by definition, $G^c= \cup_i\, G^c_i$. We have
\begin{align*}
G^c_i  &\subset  \bigcup\limits_{s\in[n]} \Big\{ \{ \mu_i \geq  \Mu_{i,s} + \epsilon_{i,s}  \}  \cup  \{  \mu_i \leq \Mu_{i,s} - \epsilon_{i,s} \} \Big\}\\
&= \bigcup\limits_{s\in[n]} \big\{ |  \mu_ i- \Mu_{i,s}|
 \geq \epsilon_{i,s}
  \big\}.
\end{align*}
Since we assume sub-Gaussian bandits, the tail decay can be bounded using  the Cramer-Chernoff bound, which gives
$$
\PP(| \mu_ i- \Mu_{i,s}|
 \geq \epsilon_{i,s} )  \leq 2\exp\Big(- \frac{s\epsilon^2_{i,s}}{2}\Big).
$$
By definition of $\epsilon_{i,s}$, we have 
\begin{equation*}
\PP(G_i^c ) \leq  \PP\Big( \bigcup\limits_{s\in[n]} \big\{ | \mu_ i- \Mu_{i,s}|
 \leq \epsilon_{i,s} \big\} \Big) 
 \leq \sum_{s=1}^n \PP\Bigg(  | \mu_ i- \Mu_{i,s}|
 \leq \sqrt{\frac{2 \log\frac{1}{\delta}}{s}} \,\Bigg)  \leq 2n\delta .
\end{equation*}
Using this estimate we have 
\begin{equation}\label{reg-1}
\EX[\mathbb{I}\{G^c\}T_{i} (n)] \leq n \PP(G^c) \leq 2n^2\delta k\leq 2k
\end{equation}
using the fact that $\delta=\frac{1}{n^2}$.\\

Next, we prove fact 2. We note that on the event $G$, $\text{argmin}_i\, \Mu_{i,T_i(t)} = \mu_1 $ is guaranteed (since $\epsilon_{i,T_i(t)} \leq  \frac{\Delta_1 -\gamma}{2} $ for all $i$, when $T_i(t) \geq  \Big\lceil \frac{16 \log n }{ (\Delta_i -\gamma)^2}\Big\rceil$). So, we need only consider the probability of transitioning from the first arm to any other sub-optimal arm. To avoid notational clutter, let $u_i =T_i(t)$ in what follows. For any round $t$ and $i \neq 1$,
\begin{align*}
    \exp\Big( -\frac{(\Mu_{i,u_i} - \Mu_{1,u_1})^+}{T_t} \Big) &=   \exp\Big( -\frac{\Mu_{i,u_i} - \Mu_{1,u_1}}{T_t} \Big) \\
    &\leq  \exp\Big( \frac{\mu_{1} + \epsilon_{1,u_1} - \mu_{i} + \epsilon_{i,u_i}}{T_t} \Big).
\end{align*}
We use the fact, on the event $G$, $\Mu_{1,s} \leq \mu_1 +  \epsilon_{1,s} $ and $\Mu_{i,s} \geq \mu_i - \epsilon_{i,s} $. Furthermore, by our choice of $\epsilon_{i,s}$, we have 
\begin{align*}
\epsilon_{i,u_i} &\leq \frac{\Delta_i -\gamma}{2}. 
\end{align*}
So, at any round $t$, we have the probability of transitioning to a sub-optimal arm is bounded by: 
\begin{align*}
    \exp\Big( -\frac{\Mu_{i,u_i} - \Mu_{1,u_1}}{T_t} \Big) &\leq  \frac{\exp( \frac{\epsilon_{1,u_1} + \epsilon_{i,u_i}}{T_t} )}{\exp( \frac{\mu_{i}  - \mu_{1}  }{T_t}) }\\
    &= \frac{t^\frac{\epsilon_{1,u_1} + \epsilon_{i,u_i}}{\gamma}} {t^{\frac{\mu_i -\mu_1}{\gamma}}}.  \\
    &\leq \frac{t^{\frac{\Delta_i -\gamma}{2\gamma}+ \frac{\Delta_2 -\gamma}{2\gamma}}}{t^{\frac{\Delta_i}{\gamma}}}\\
    &= \frac{t^{\frac{\Delta_i +\Delta_2}{2\gamma}-1}}{t^{\frac{\Delta_i}{\gamma}}}\\
    &\leq \frac{1}{t}. 
\end{align*}
The above estimate can be used to bound $\EX[T_i(n)]$ on $G$. We have, on event $G$, the following holds 
\begin{equation}\label{reg-2}
\EX[\mathbb{I}\{G\} T_i(n)]  \leq  \Big\lceil \frac{16 \log n }{ (\Delta_i - \gamma)^2}\Big\rceil + \frac{1}{k-1}\sum_{t\in [n]} \frac{1}{t} 
\leq \Big\lceil \frac{16 \log n }{ (\Delta_i - \gamma)^2}\Big\rceil + \frac{\log n}{k-1}.
\end{equation}
It remains to bring together our estimates of $\EX[T_i(n)]$ on events $G$ and $G^c$ to finish the proof. Using (\ref{reg-1}) and (\ref{reg-2}) in (\ref{mainbanditregret}), we have
\begin{multline*}
R_n = \sum_i \Delta_i \EX[T_i(n)] \leq  \sum_i \Delta_i \Big(2k+ \Big\lceil \frac{16 \log n }{ (\Delta_i - \gamma)^2}\Big\rceil + \frac{\log n}{k-1} \Big)
\leq  \Big(2k+ \frac{ \log n}{k-1}\Big) \sum_{i=2}^{k} \Delta_i + \sum_{i=2}^k \frac{\Delta_i  \times 16  \log n }{ (\Delta_i - \gamma)^2}. 
\end{multline*}
\end{proof}

We have stated Algorithm 2 with $\gamma := \frac{\Delta_2}{2}$ in Step(a). The regret bound obtained above can be turned into one that is independent of the reciprocal of the sub-optimality gaps, $\Delta_i$. A straightforward calculation in the above regret bound yields (see e.g. Theorem 7.2, \cite{tor}):
$$
R_n \leq  16 \sqrt{nk \log n}  + \Big(2k + \frac{ \log n}{k-1}\Big) \sum_{i=1}^{k} \Delta_i. 
$$

\section*{Appendix A. Numerical Experiments}
In this appendix, we provide empirical confirmation of the efficiency of the proposed algorithm for solving the stochastic MAB. An instance of the bandit problem is characterized by $k$ (number of arms) and the reward distributions (assumed to be normal with variance $\sigma^2$). We consider a reward maximization problem with the reward means calculated according to:
$$
\mu_{a}=
\begin{cases}
0, \,\,\text{if } a =a^* \\
-\Delta_{\text{min}} - | \mathcal{N}(0,0.1)|, \,\,\text{if } \,a \in \A/\{a^*\}
\end{cases}
$$
The value $\Delta_{\text{min}}$ has been set to $0.1$ for all experiments performed here.

We briefly describe the methodology of the algorithms we use for performance comparison:

\textit{(i) $\epsilon$-Greedy:} At round $t$,  the probability $P_a(t)$ of selecting arm $a$ is given by:
$$
P_a(t)= \begin{cases}
1- \epsilon + \frac{\epsilon}{k},\,\,\,\text{if} \,\, a = \argmax_{a' \in \A} \hat{\mu}_{a'}(t),\\
\frac{\epsilon}{k}\,\,\,\text{otherwise.}
\end{cases}
$$

\textit{(ii) UCB:} The UCB family of algorithms incorporate the idea of optimism in the face of uncertainty to determine the policy. Initially, each arm is played once and subsequently for any round $t$, the algorithm greedily selects arm $a_t$ according to:
$$
a_t \in \argmax_{a'\in\A} \Bigg\{ \hat{\mu}_{a'}(t) + \sqrt{ \frac{2 \ln t}{T_{a'}(t)}}\,\Bigg\}.
$$

\textit{(iii) Boltzmann Exploration/Softmax:} This algorithm is a softmax method where the probability of picking the arm is decided by the Boltzmann distribution, i.e. the probability of selecting an arm is proportional to its current empirical mean: 
$$
P_a(t+1) := \frac{e^{\frac{\hat{\mu}_a(t)}{\tau}}}{\sum_{a' \in \A}  e^{\frac{\hat{\mu}_{a'}(t)}{\tau}}},
$$ 
where $\tau$ controls the randomness of the choice.

The results have been plotted for two criterion: (i) Fraction of optimal arm plays (ii) Regret accumulated till time $n$. The parameters for all the comparison algorithm have been set according to \cite{krup}. As is evident from the plotted results, the performance of the tested algorithms varies to a great degree depending on $k$ and $\sigma^2$. The broad conclusion that one can draw about the proposed algorithm is that \emph{it has the most robust performance across all the metrics we have tested against.} For large variances, the performance of the algorithms do not differ to any significant degree except for the case of 2-armed bandit where the UCB algorithm dominates for large variances but the performance degrades significantly for low variances (this is consistent with the observations of \cite{krup}). 
\newpage
\begin{table}[h]
\begin{tabular}{c@{}c@{}c}

  \includegraphics[width=0.5\textwidth]{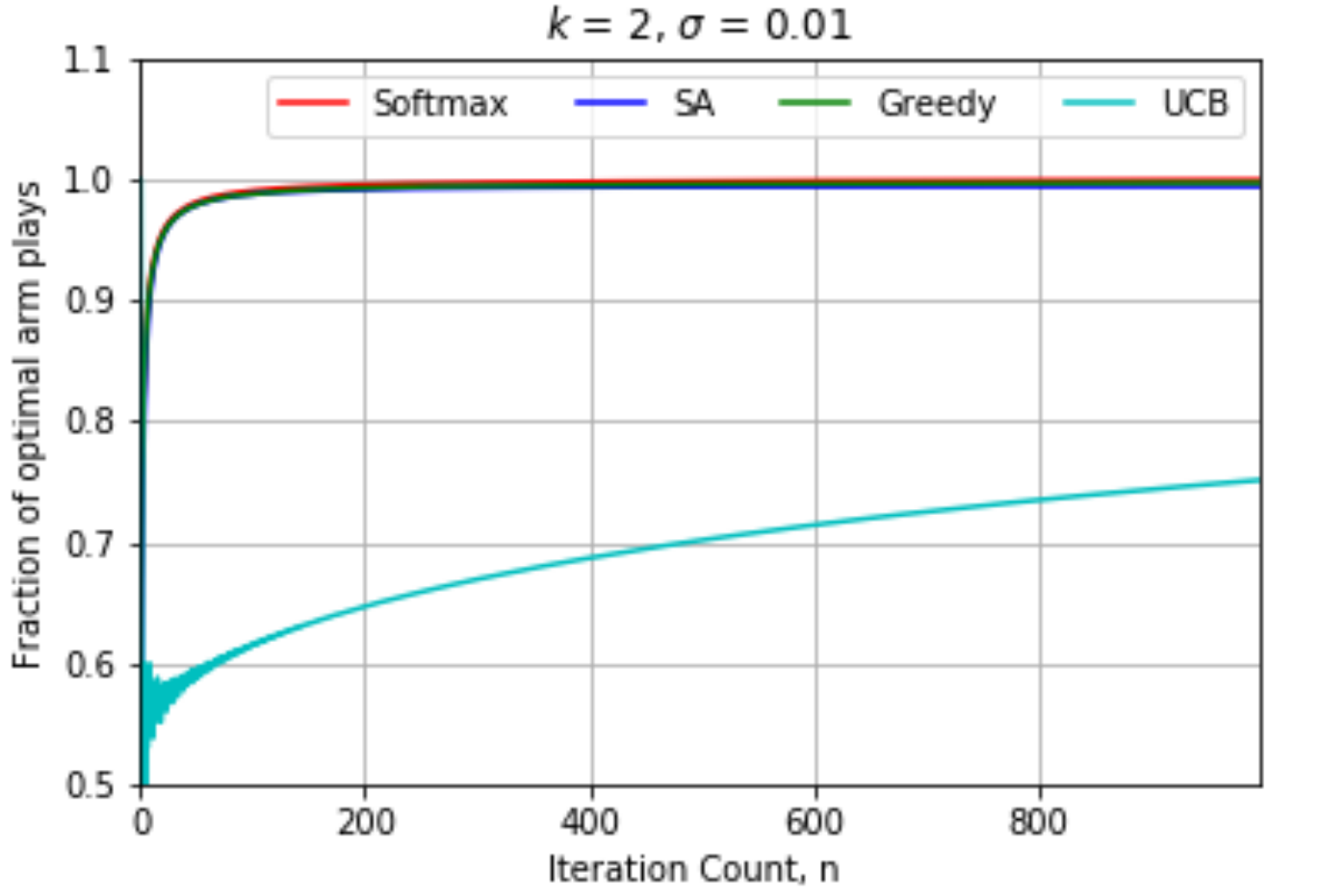}&
  \includegraphics[width=0.5\textwidth]{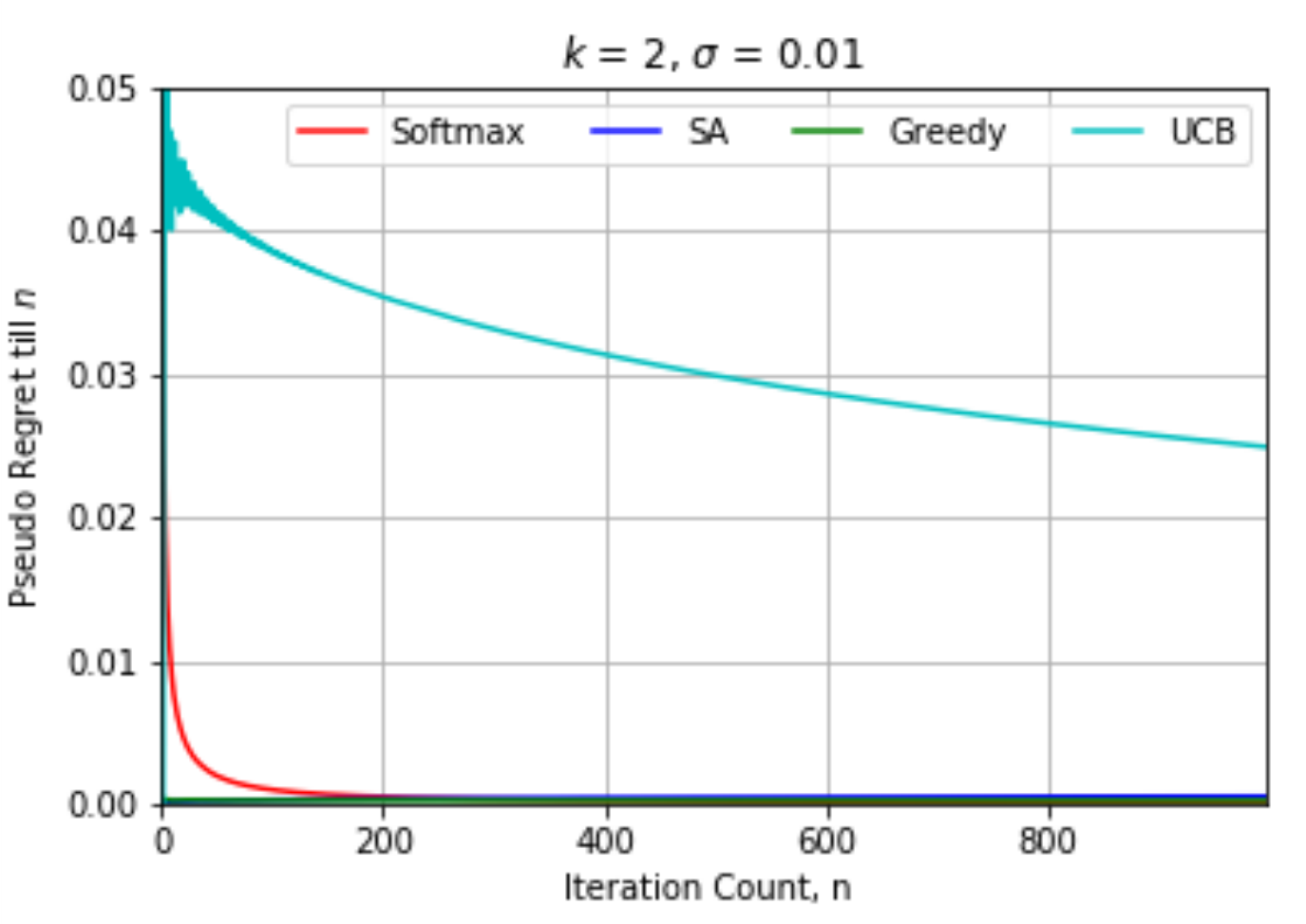}\\
  
     \includegraphics[width=0.5\textwidth]{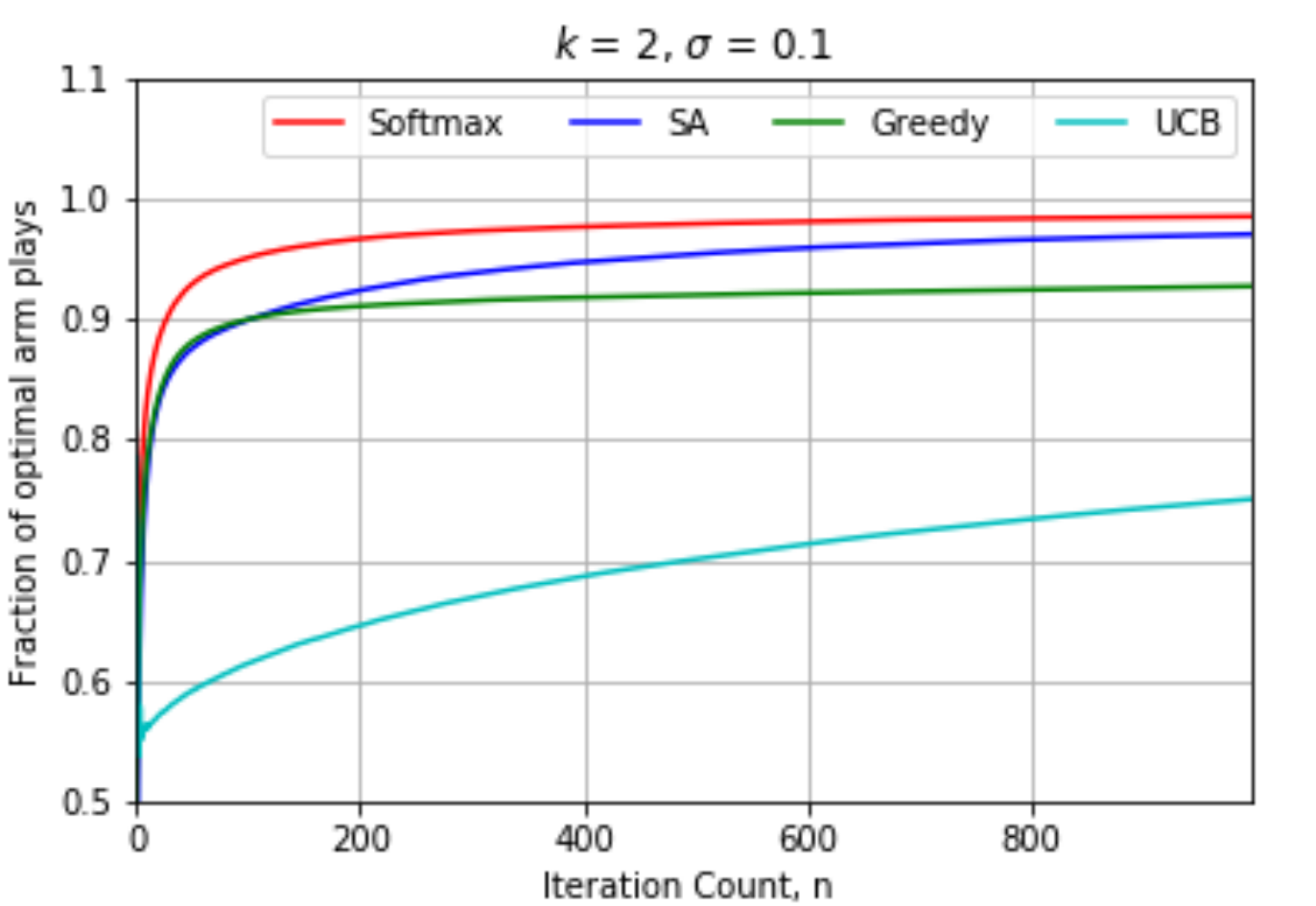}&
  \includegraphics[width=0.5\textwidth]{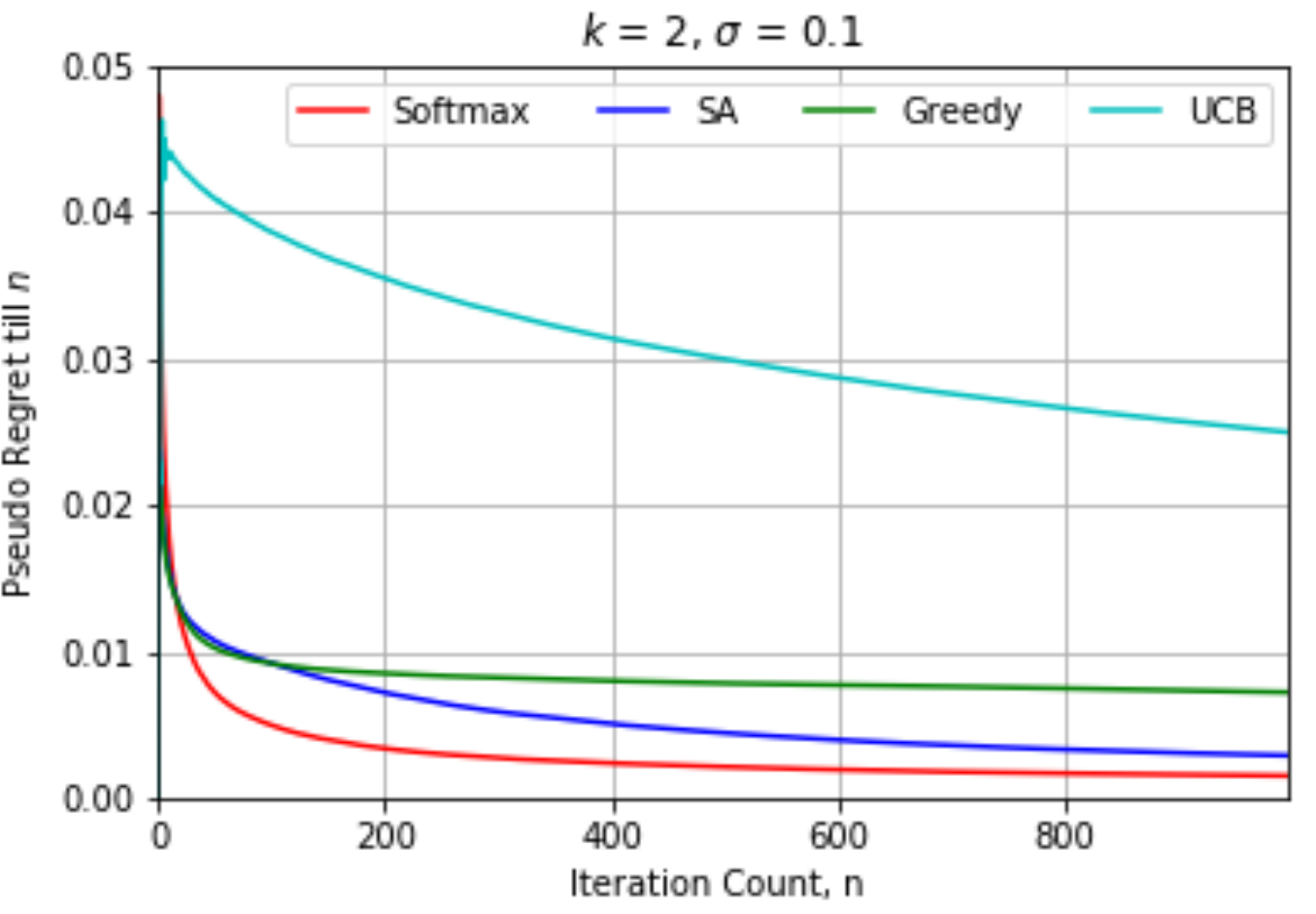}\\
  
   \includegraphics[width=0.5\textwidth]{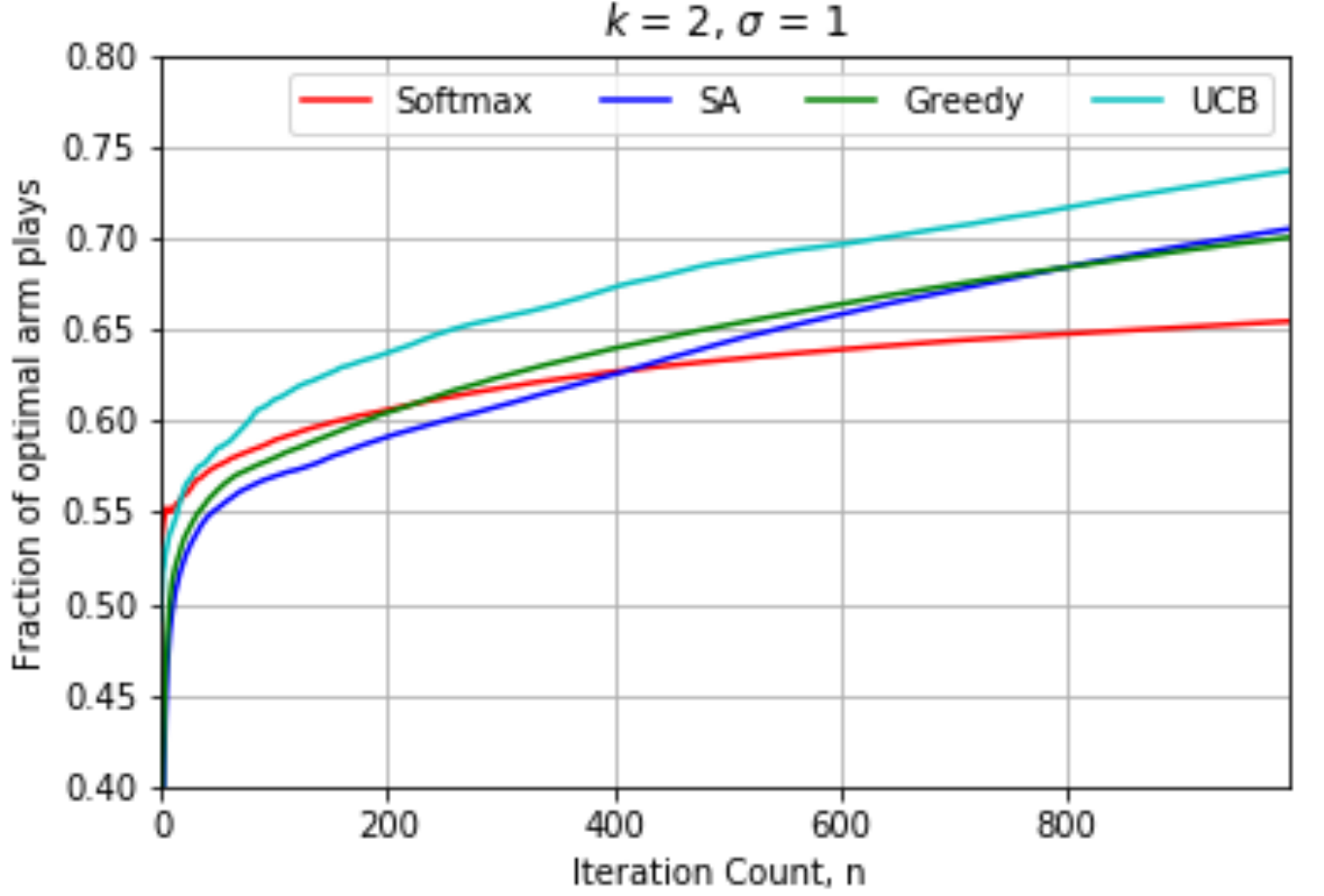}&
  \includegraphics[width=0.5\textwidth]{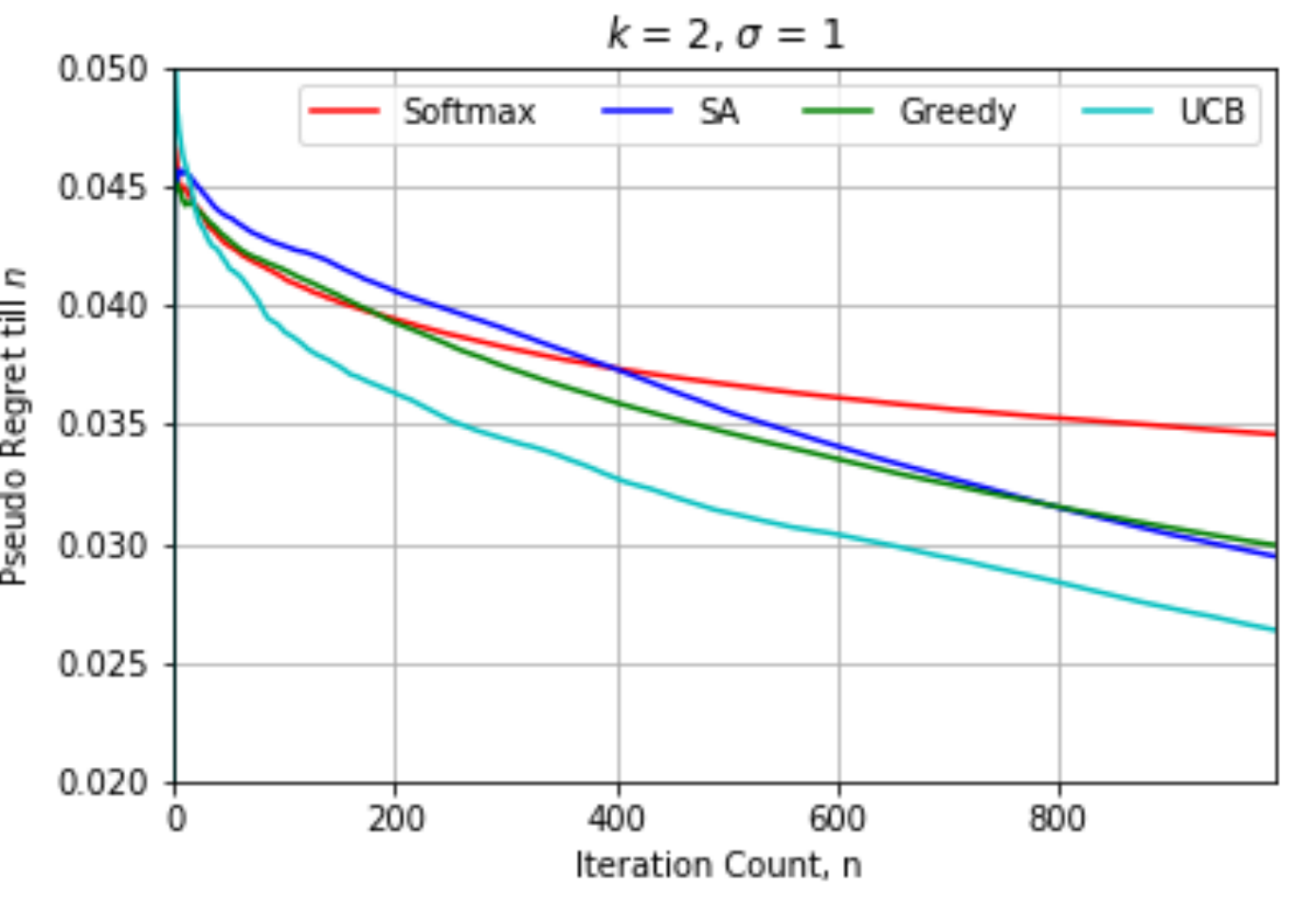}\\

\end{tabular}
\centering
\caption*{\textbf{Parameters ($k=2,\, \sigma \in\{0.01,0.1,1\}, \,\Delta_{\text{min}}=0.1$)}:\\
\textbf{$\sigma=0.01$} :  $\text{SA, }\gamma^{-1}= 0.002\,;\,\text{Softmax, } \tau = 0.001\,;\,\epsilon\text{-greedy, } : \epsilon = 0.005$.\\
 \textbf{$\sigma=0.1$} :  $\text{SA, }\gamma^{-1}= 0.01\,;\,\text{Softmax, } \tau = 0.01\,;\,\epsilon\text{-greedy, } : \epsilon = 0.001$.\\
  \textbf{$\sigma=1$} :  $\text{SA, }\gamma^{-1}= 0.2\,;\,\text{Softmax, } \tau = 0.1\,;\,\epsilon\text{-greedy, } : \epsilon = 0.05$.\\}
\end{table}
\newpage

\begin{table}[h]
\begin{tabular}{c@{}c@{}c}

  \includegraphics[width=0.48\textwidth]{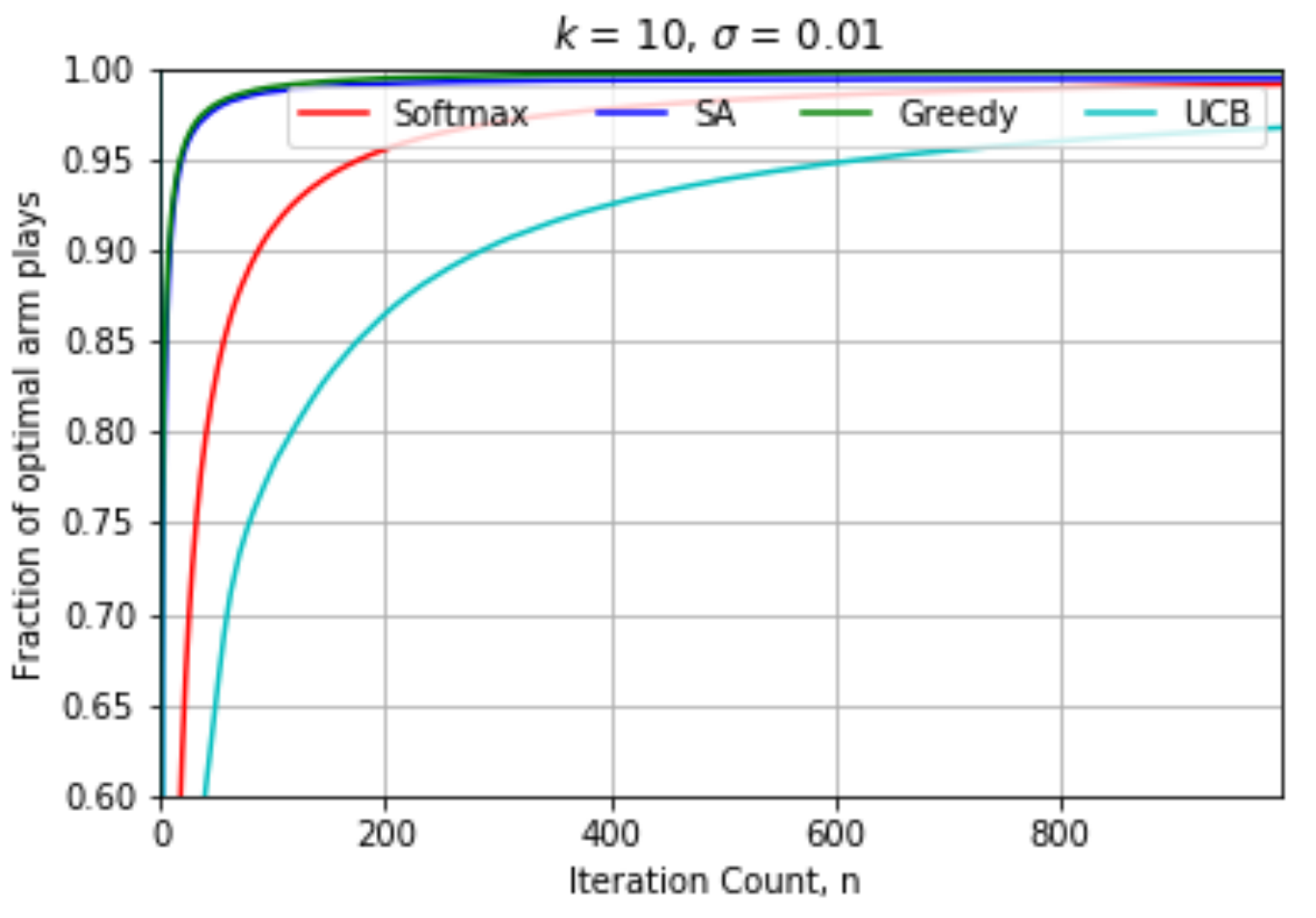}&
  \includegraphics[width=0.5\textwidth]{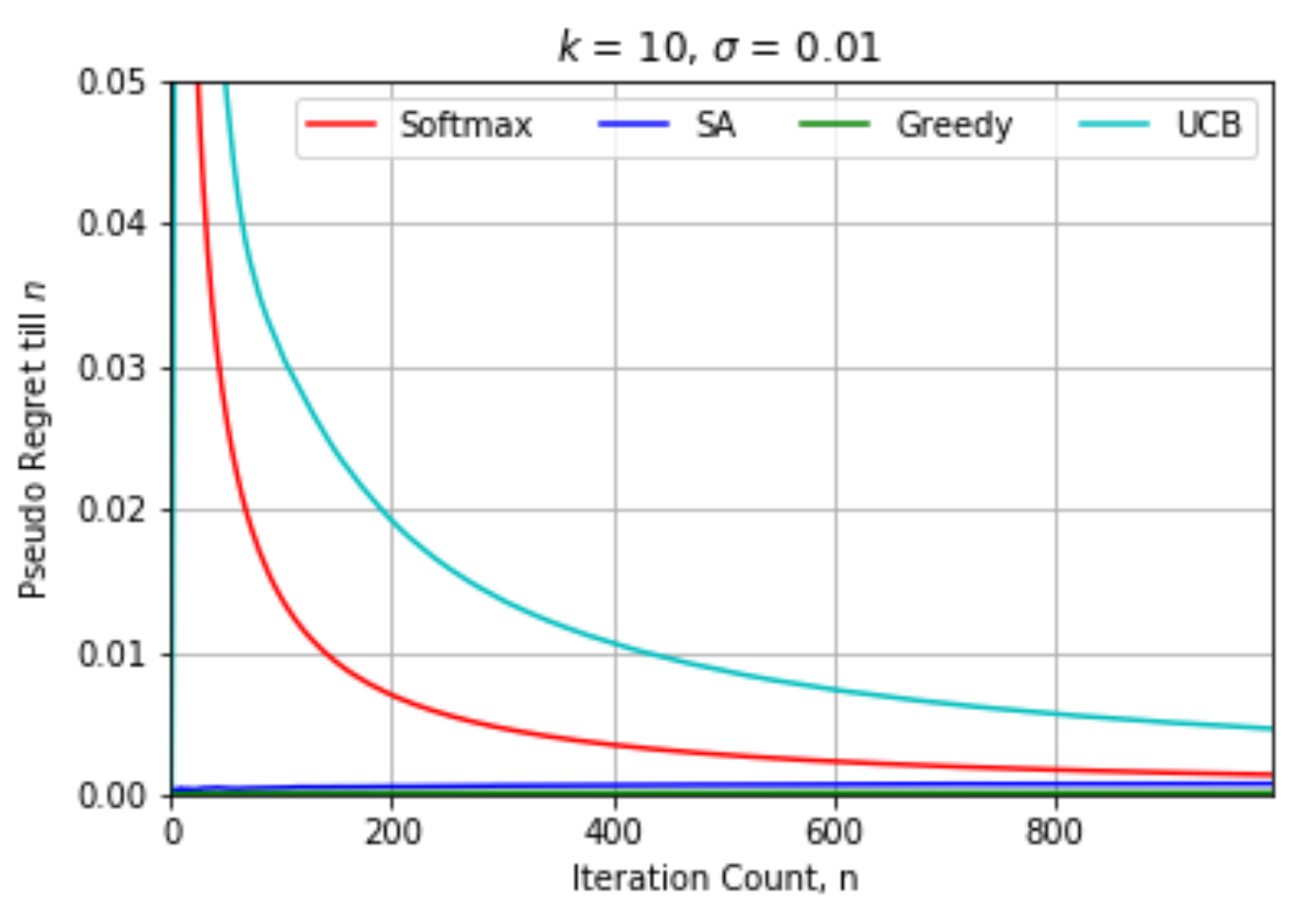}\\
  caption a & caption b\\
    
     \includegraphics[width=0.5\textwidth]{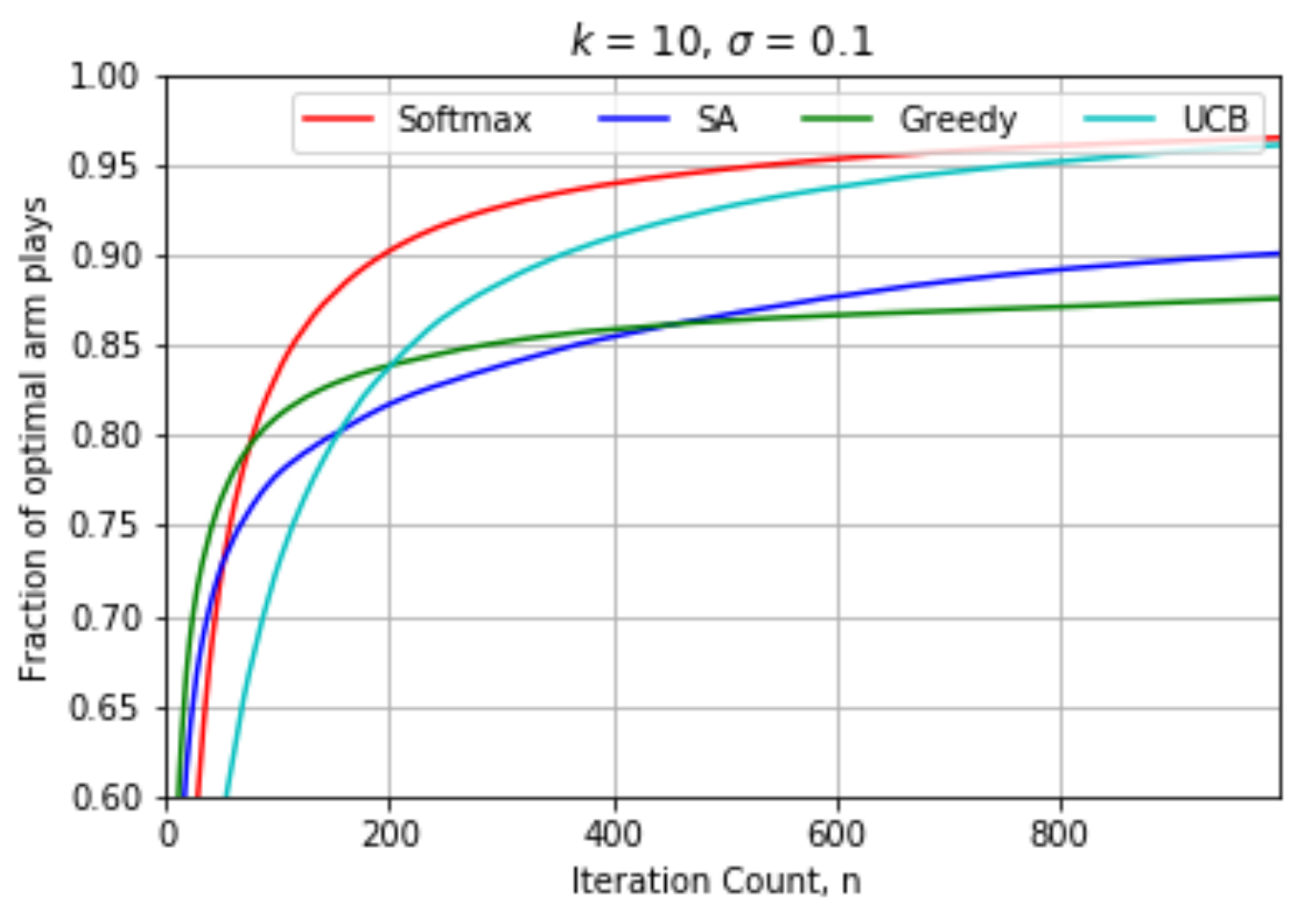}&
  \includegraphics[width=0.5\textwidth]{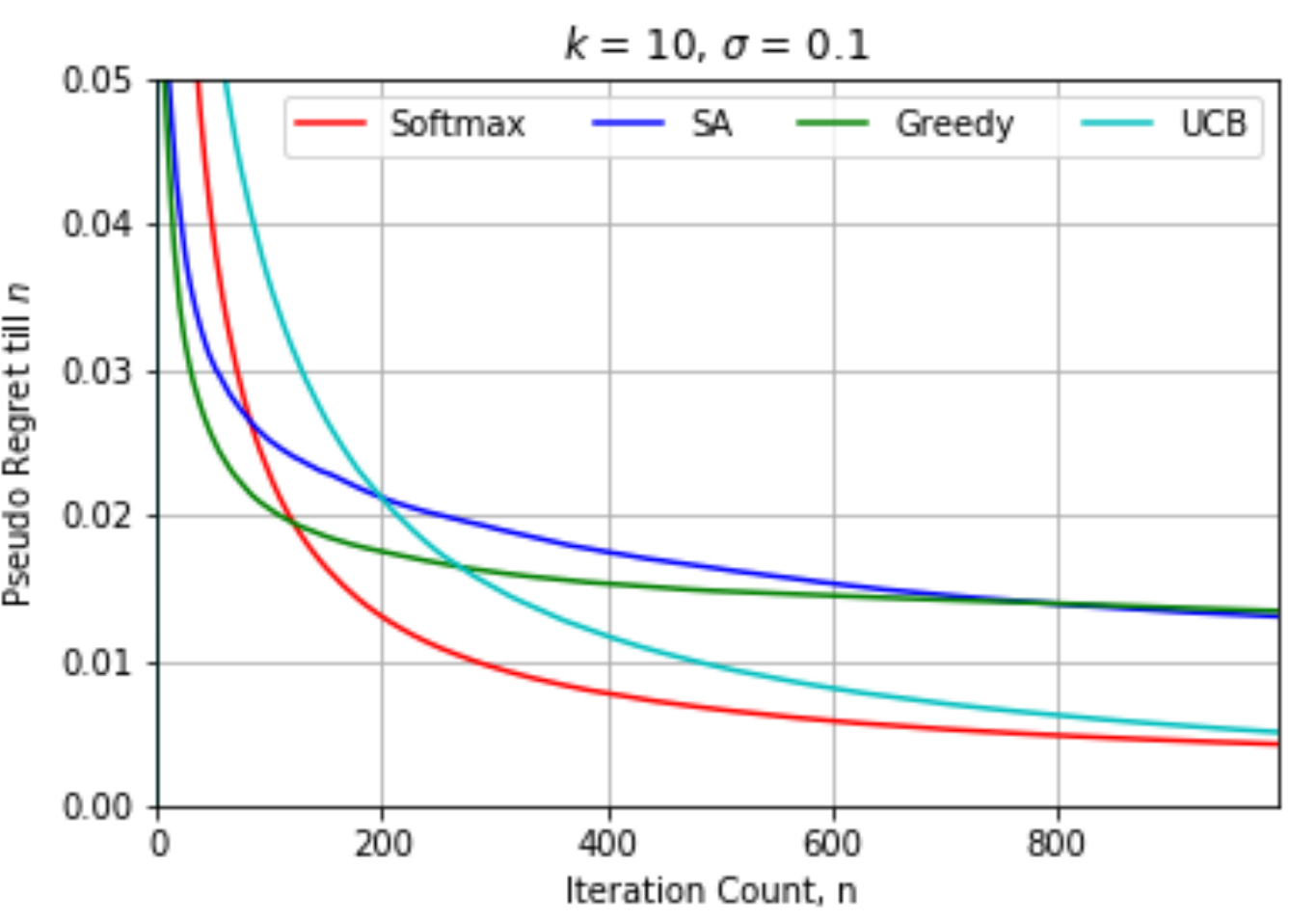}\\
  
   \includegraphics[width=0.5\textwidth]{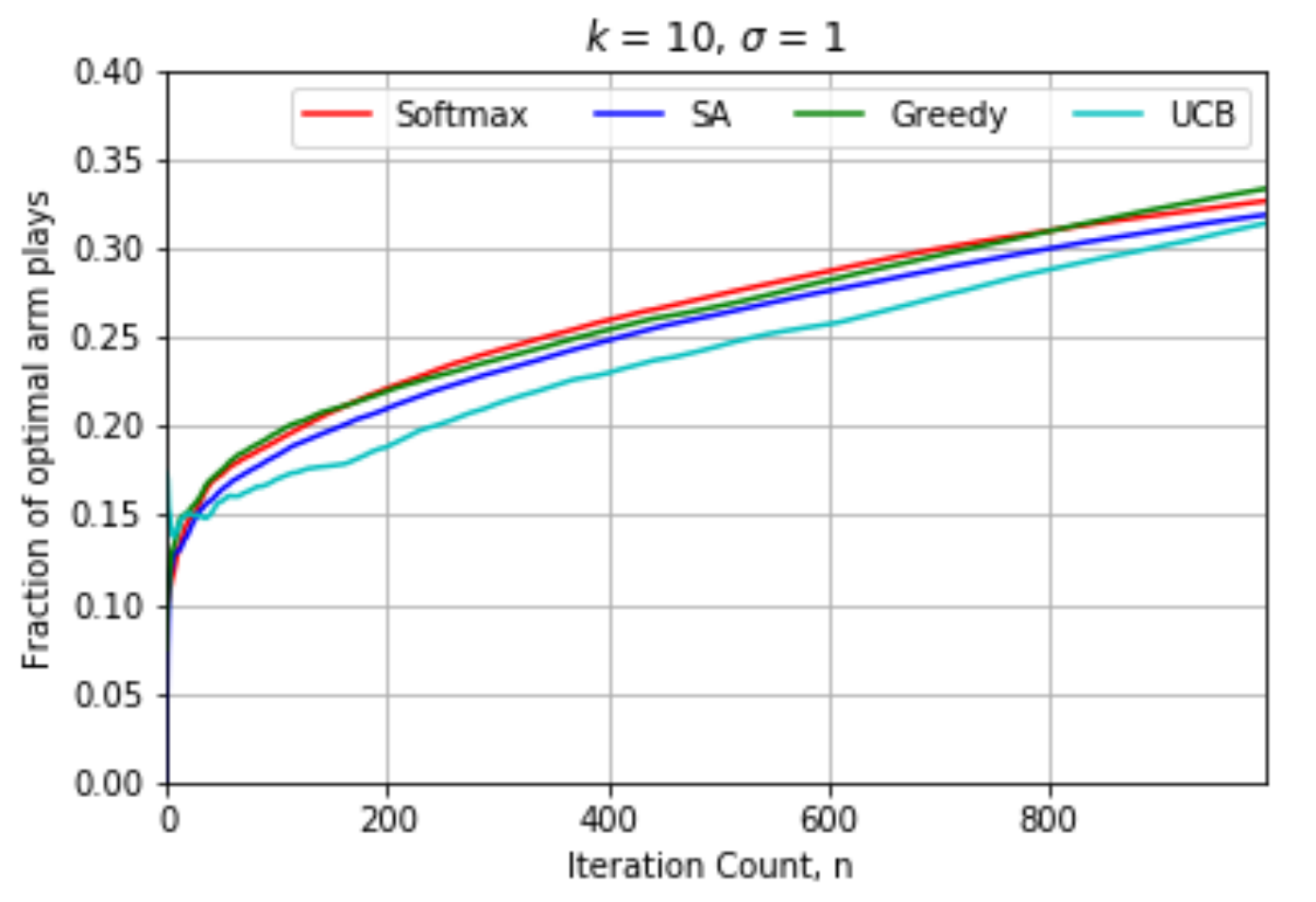}&
  \includegraphics[width=0.5\textwidth]{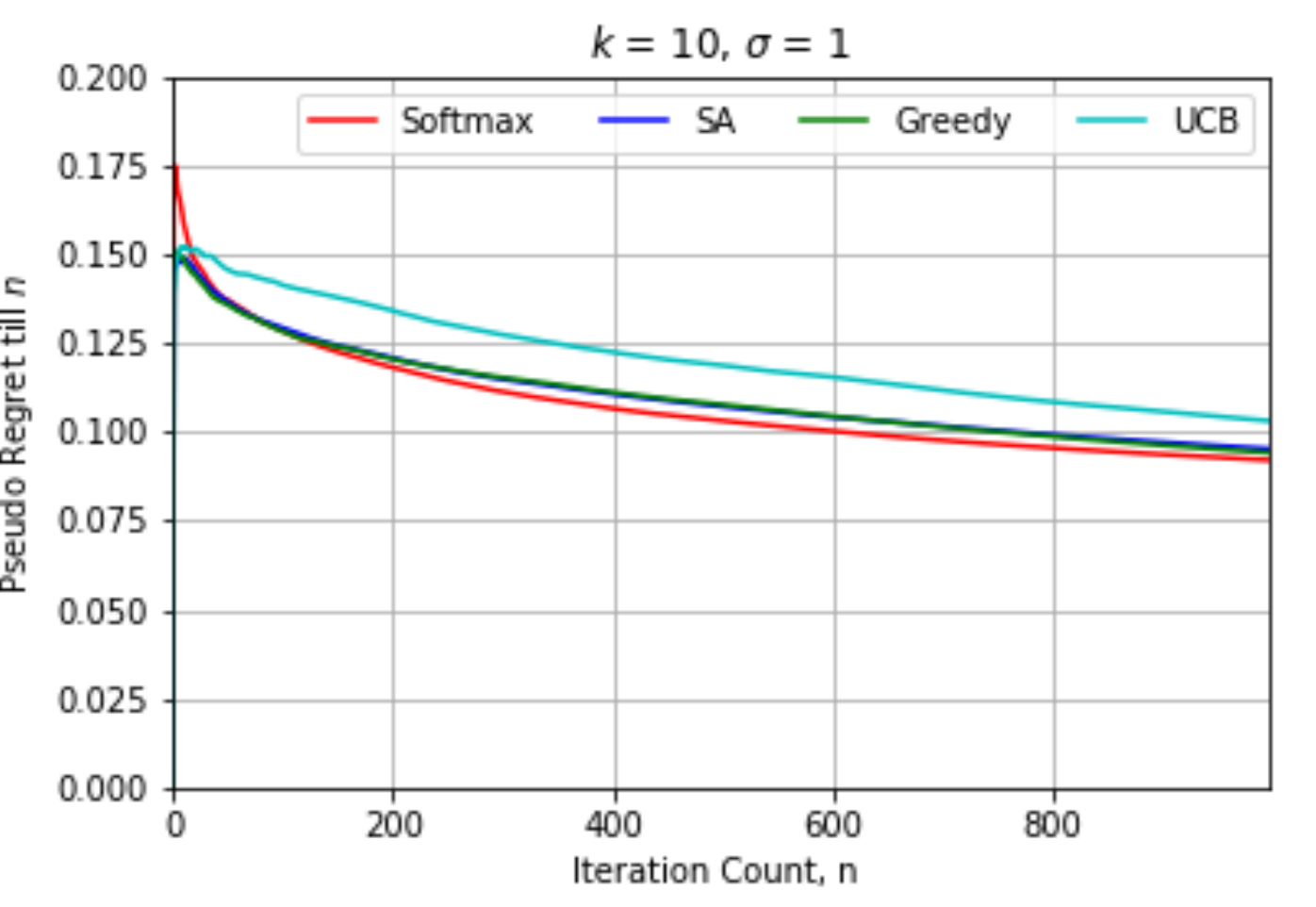}\\

\end{tabular}
\centering
\caption*{\textbf{Parameters ($k=10,\, \sigma \in\{0.01,0.1,1\}, \,\Delta_{\text{min}}=0.1$)}:\\
\textbf{$\sigma=0.01$} :  $\text{SA, }\gamma^{-1}= 0.001\,;\,\text{Softmax, } \tau = 0.001\,;\,\epsilon\text{-greedy, } : \epsilon = 0.001$.\\
 \textbf{$\sigma=0.1$} :  $\text{SA, }\gamma^{-1}= 0.01\,;\,\text{Softmax, } \tau = 0.01\,;\,\epsilon\text{-greedy, } : \epsilon = 0.005$.\\
  \textbf{$\sigma=1$} :  $\text{SA, }\gamma^{-1}= 0.05\,;\,\text{Softmax, } \tau = 0.05\,;\,\epsilon\text{-greedy, } : \epsilon = 0.1$.\\}
\end{table}
\newpage
\begin{table}[h]
\begin{tabular}{c@{}c@{}c}

  \includegraphics[width=0.5\textwidth]{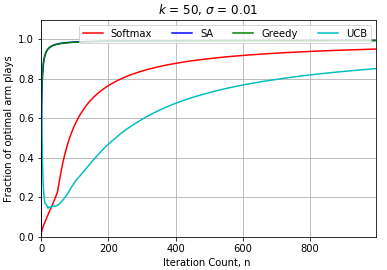}&
  \includegraphics[width=0.5\textwidth]{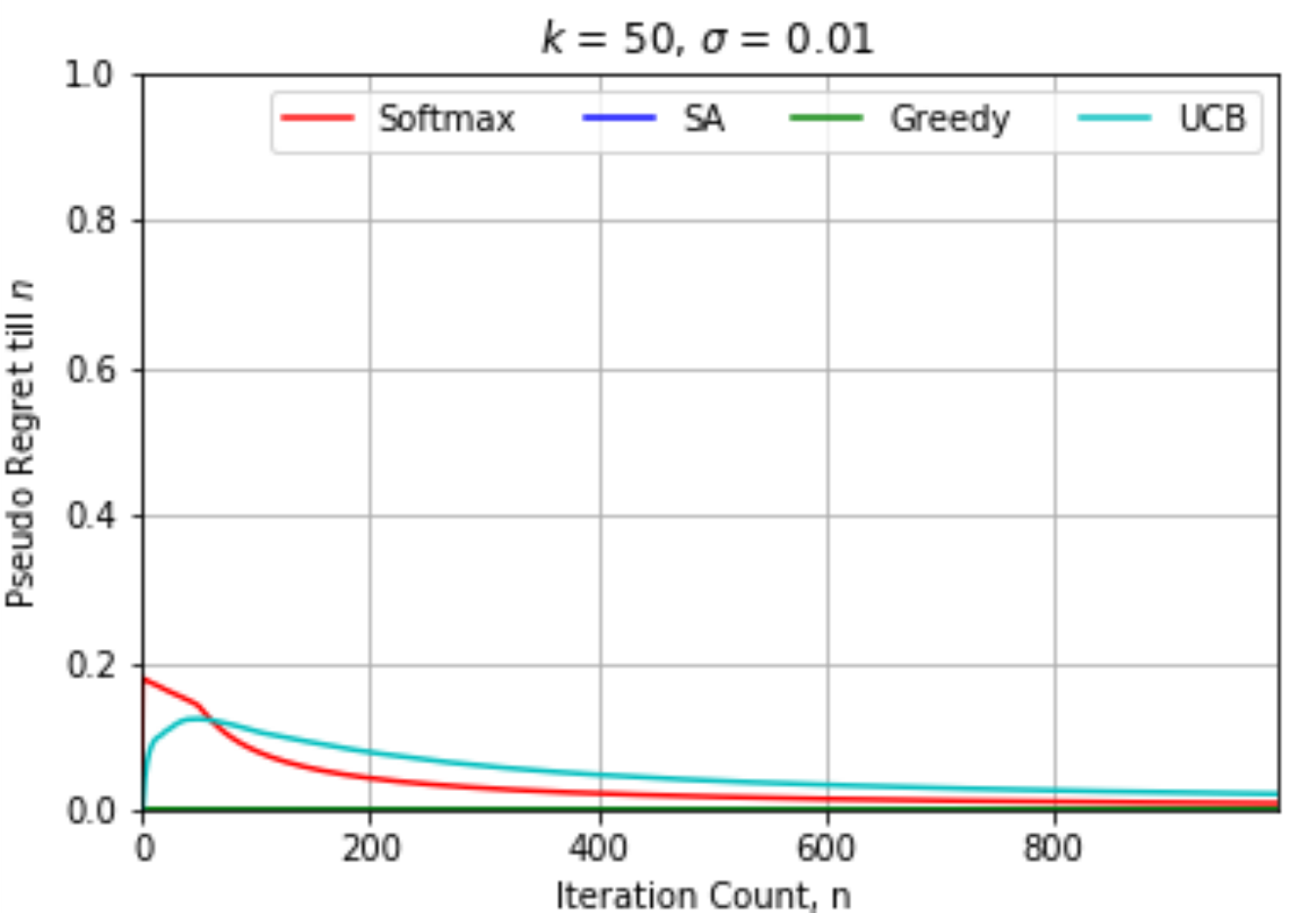}\\
    
     \includegraphics[width=0.49\textwidth]{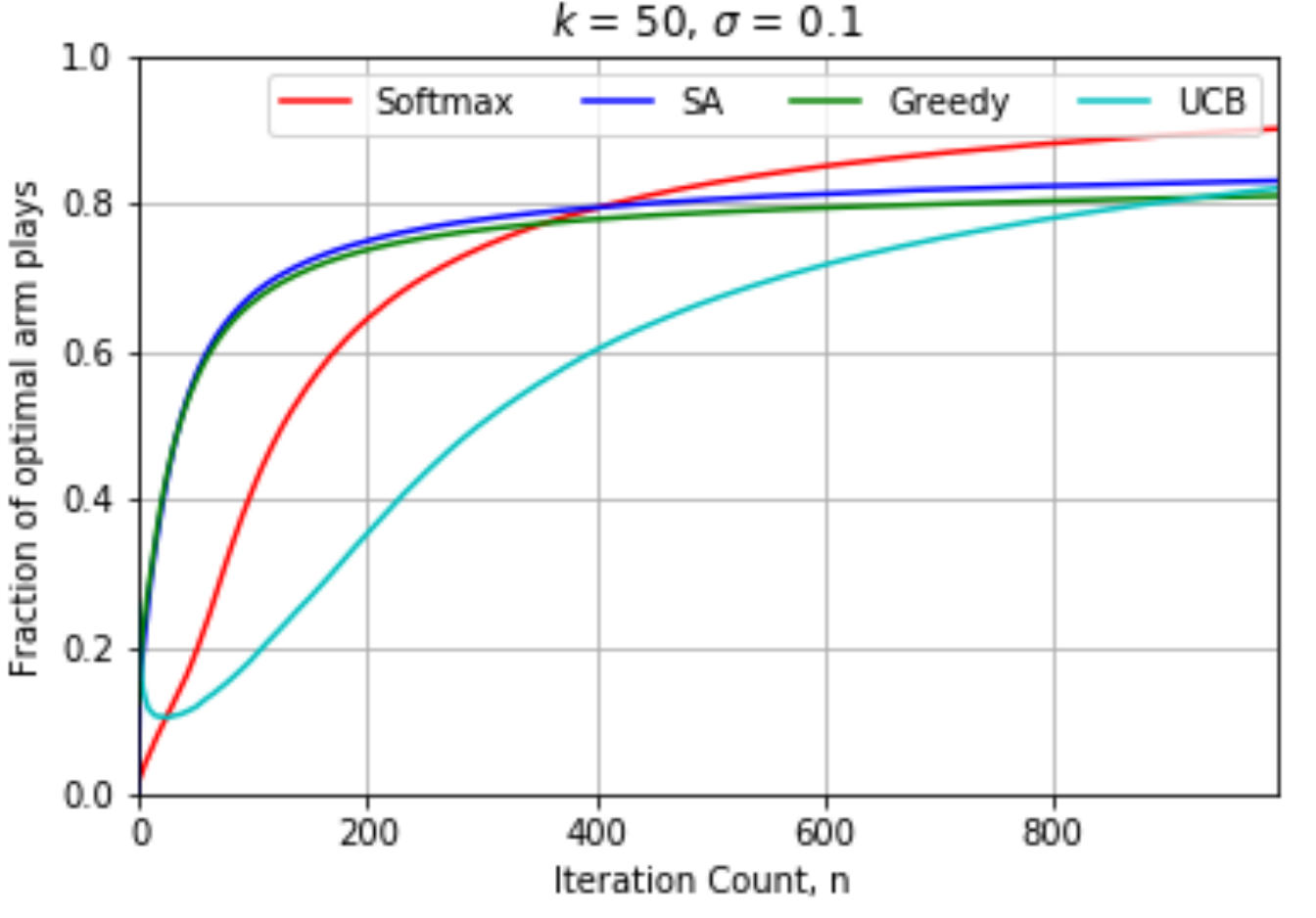}&
  \includegraphics[width=0.5\textwidth]{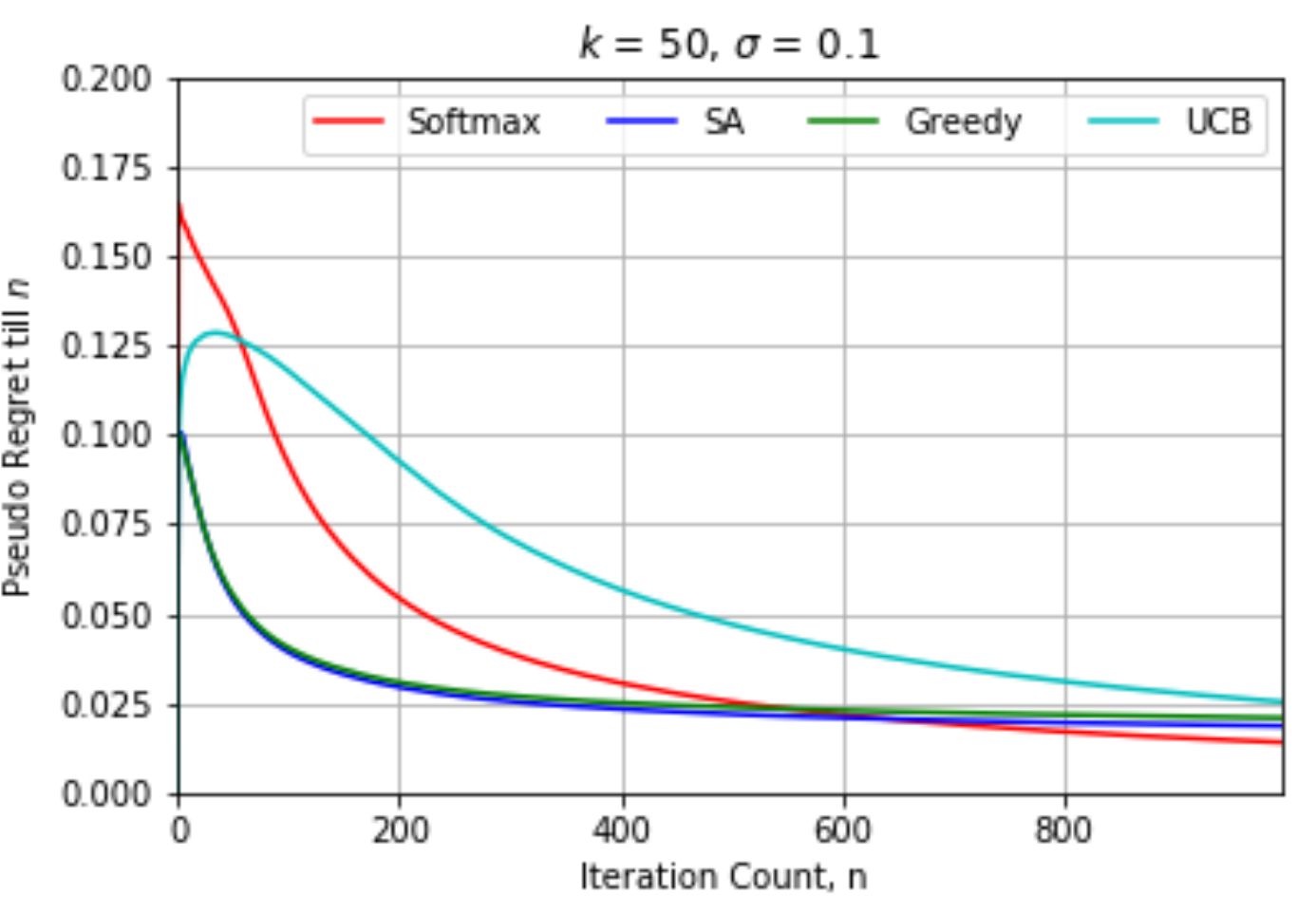}\\
  
   \includegraphics[width=0.5\textwidth]{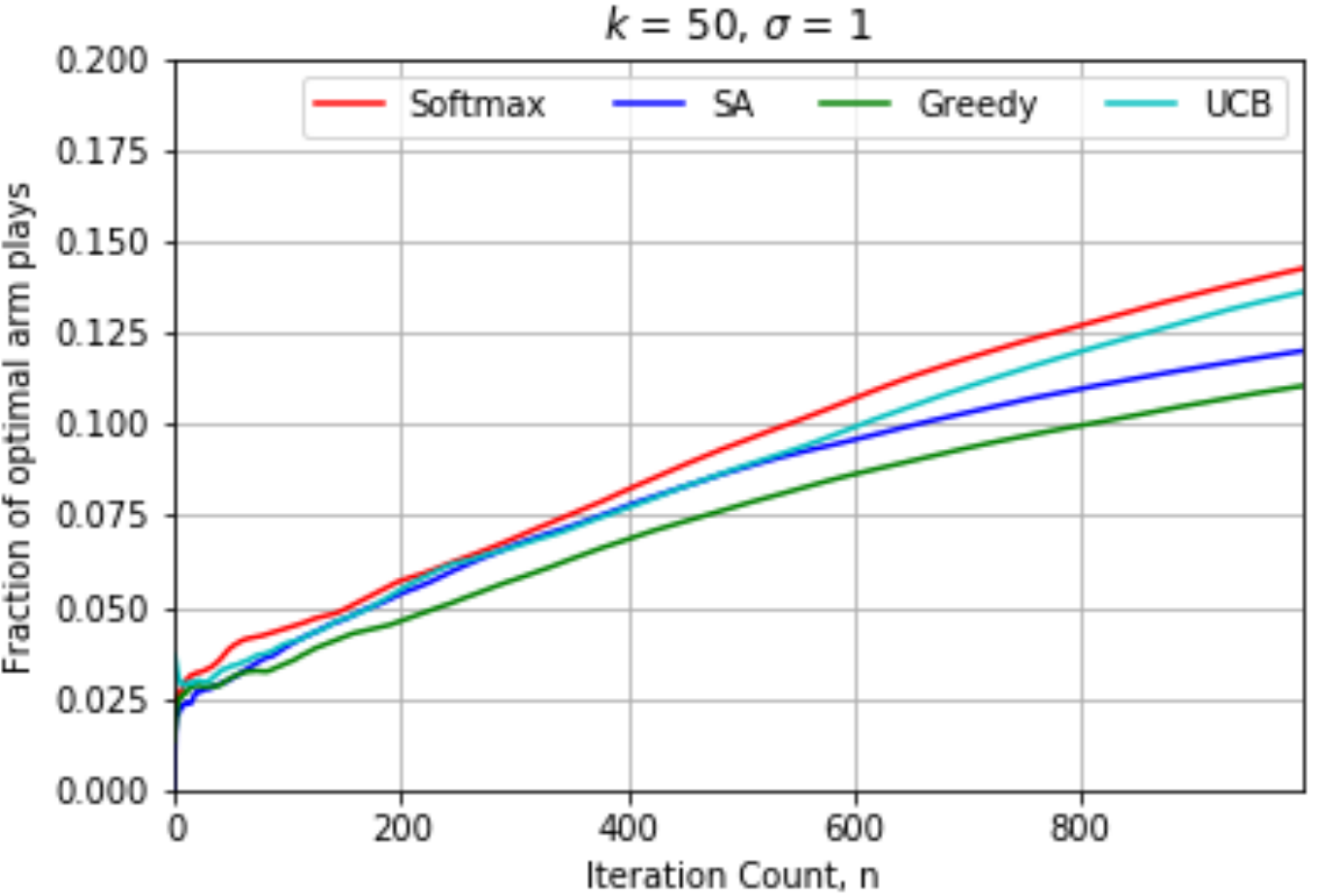}&
  \includegraphics[width=0.5\textwidth]{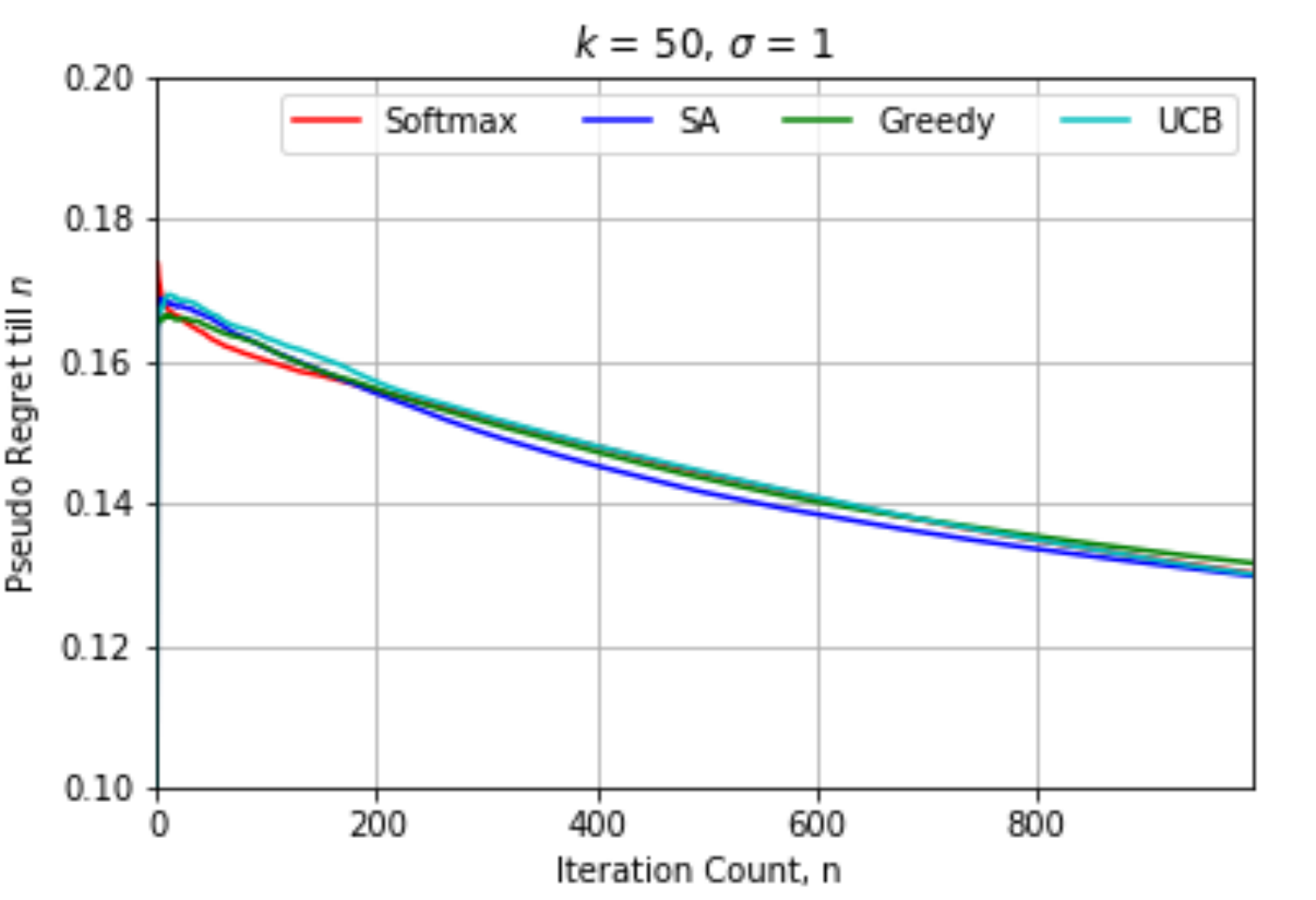}\\

\end{tabular}
\centering
\caption*{\textbf{Parameters ($k=50,\, \sigma \in\{0.01,0.1,1\}, \,\Delta_{\text{min}}=0.1$)}:\\
\textbf{$\sigma=0.01$} :  $\text{SA, }\gamma^{-1}= 0.001\,;\,\text{Softmax, } \tau = 0.001\,;\,\epsilon\text{-greedy, } : \epsilon = 0.005$.\\
 \textbf{$\sigma=0.1$} :  $\text{SA, }\gamma^{-1}= 0.01\,;\,\text{Softmax, } \tau = 0.01\,;\,\epsilon\text{-greedy, } : \epsilon = 0.005$.\\
  \textbf{$\sigma=1$} :  $\text{SA, }\gamma^{-1}= 0.01\,;\,\text{Softmax, } \tau = 0.007\,;\,\epsilon\text{-greedy, } : \epsilon = 0.01$.\\}
\end{table}

\section*{Appendix B. List of Notations}

\begin{table}
 \begin{tabular}{ ?p{5cm}?p{10cm}? }
\specialrule{1pt}{1pt}{1pt}
 \centering
Notation& $\,\,\,\,\,\,\,\,\,\,\,\,\,\,\,\,\,\,\,\,\,\,\,\,\,\,\,\,\,\,\,\,\,\,\,$Description  \\
 \specialrule{1pt}{1pt}{1pt}
$\|x\|,\,x\in \mathbb{R}^k$ & $\sum_{i \in k} |x_i|$.\\
 \specialrule{1pt}{1pt}{1pt}
 $\|P\|,\,P\in \mathbb{R}^{m \times n} $ & $\max_{1\leq i \leq m}\sum_{j = 1}^n |p_{ij}|$.\\
 \specialrule{1pt}{1pt}{1pt}
  $ \|f \|_{\pi},\,\, \,f : \A \to \mathbb{R}^{k} $ & $\sqrt{\langle|f|^2 \rangle_{\pi}} \,\,\text{ where }\,\, \langle g \rangle_{\pi} := \sum_{i \in \A} g_i \pi_i $.\\
 \specialrule{1pt}{1pt}{1pt}
 $\A $ & Arm Set.\\
 \specialrule{1pt}{1pt}{1pt}
 $\G =\{\A,\E\}$ & Graph with edge set $\E$.\\
 \specialrule{1pt}{1pt}{1pt}
 $n $ & Total number of rounds.\\
 \specialrule{1pt}{1pt}{1pt}
 $k $ & Total number of arms.\\
 \specialrule{1pt}{1pt}{1pt}
 $\mu $ & Loss vector with component $\mu_i$ denoting the loss of arm $i$.\\
  \specialrule{1pt}{1pt}{1pt}
 $\mathcal{E}^{k}_{\text{SG}}(\sigma^2)$ & Sub Gaussian Bandits with variance $\sigma^2$ and $|\A|=k$.\\
 \specialrule{1pt}{1pt}{1pt}
 $\hat{\mu}_{i}(t),\,\hat{\mu}_{i,T_a(t)} $ & Empirical average of the loss at arm $i$ at round $t$ or with $T_a(t)$ samples.\\
 \specialrule{1pt}{1pt}{1pt}
 $\Delta_i $ & $\mu_i -\mu_{a^*}$.\\
 \specialrule{1pt}{1pt}{1pt}
 $\nu(n)$ & Probability distribution of selecting the arms (i.e. policy) at time $n$.\\
 \specialrule{1pt}{1pt}{1pt}
 $T_a(t)$ & Total number of pulls of arm $a$ at round $t$.\\
 \specialrule{1pt}{1pt}{1pt}
 $T_p $ & Temperature/Cooling schedule during round $p$.\\
 \specialrule{1pt}{1pt}{1pt}
 $\gamma^* $ & Critical Depth (Definition 3).\\
 \specialrule{1pt}{1pt}{1pt}
 $\beta(t) $ & $1/T_t$.\\
 \specialrule{1pt}{1pt}{1pt}
 $P(t)$ & Transition probability at round $t$, see (\ref{MC}).\\
 \specialrule{1pt}{1pt}{1pt}
 $ \PPP(m,n) $ & $\prod_{t=m}^{n-1 }P(t)$.\\
  \specialrule{1pt}{1pt}{1pt}
 $ L $ &  $\max_{i\in \A}\max_{j \in \N(i)} |\mu_j - \mu_i|$.\\
  \specialrule{1pt}{1pt}{1pt}
  $ R $ &  $ \min_{i\in \A}\min_{j\in \N(i)} \frac{g(i,j)}{g(i)}$.\\
 \specialrule{1pt}{1pt}{1pt}
 $g(i,j)/g(i)$ & Probaiblity of selecting $j$ while in node $i$.\\
 \specialrule{1pt}{1pt}{1pt}
 $g(i)$ & $\sum_{j \in \N(i) }g(i,j)$\\
 \specialrule{1pt}{1pt}{1pt}
 $\pi(t)$ & Quasi stationary distribution of $P(t)$, see eq. (\ref{pi}).\\
 \specialrule{1pt}{1pt}{1pt}
  $\kappa(P)$ & Ergodicity co-efficient of matrix $P$.\\
 \specialrule{1pt}{1pt}{1pt}
  $ \text{Var}_{\pi}(f) $ & $ \|f- \langle f \rangle_{\pi}\|_{\pi}^2 $.\\
   \specialrule{1pt}{1pt}{1pt}
   $Q(t)$  & \textit{Q}-matrix.\\
 \specialrule{1pt}{1pt}{1pt}
   $\E_t(f,f) $ & $  \frac{1}{2} \sum_{\substack{i\in \A\\ j \neq i}} \pi_i Q_{ij}(t) (f_j-f_i)^2.$.\\
 \specialrule{1pt}{1pt}{1pt}
 $ \lambda(t) $ & $  \inf\{ \E_t(f,f)  :\,\,\text{Var}_{\pi(t)}(f) =1\}$.\\
 \specialrule{1pt}{1pt}{1pt}
  \end{tabular}
\end{table}

\bibliographystyle{ieeetr}
\bibliography{biblo}

\begin{thebibliography}{10}

\bibitem{lai}
H.~Robbins, ``Some aspects of the sequential design of experiments,'' {\em
  Bulletin of the American Mathematical Society}, vol.~58, no.~5, pp.~527--535,
  1952.

\bibitem{aue}
P.~Auer, N.~Cesa-Bianchi, and P.~Fischer, ``Finite-time analysis of the
  multiarmed bandit problem,'' {\em Machine learning}, vol.~47, no.~2-3,
  pp.~235--256, 2002.

\bibitem{bubeck}
S.~Bubeck and N.~Cesa-Bianchi, ``Regret analysis of stochastic and
  nonstochastic multi-armed bandit problems,'' {\em arXiv preprint
  arXiv:1204.5721}, 2012.

\bibitem{tor}
T.~Lattimore and C.~Szepesv{\'a}ri, {\em Bandit algorithms}.
\newblock Cambridge University Press, 2020.

\bibitem{kirk}
S.~Kirkpatrick, C.~D. Gelatt, and M.~P. Vecchi, ``Optimization by simulated
  annealing,'' {\em science}, vol.~220, no.~4598, pp.~671--680, 1983.

\bibitem{bert}
D.~Bertsimas, J.~Tsitsiklis, {\em et~al.}, ``Simulated annealing,'' {\em
  Statistical science}, vol.~8, no.~1, pp.~10--15, 1993.

\bibitem{hajek}
B.~Hajek, ``Cooling schedules for optimal annealing,'' {\em Mathematics of
  operations research}, vol.~13, no.~2, pp.~311--329, 1988.

\bibitem{mitra}
D.~Mitra, F.~Romeo, and A.~Sangiovanni-Vincentelli, ``Convergence and
  finite-time behavior of simulated annealing,'' in {\em 1985 24th IEEE
  Conference on Decision and Control}, pp.~761--767, IEEE, 1985.

\bibitem{gidas}
B.~Gidas, ``Nonstationary markov chains and convergence of the annealing
  algorithm,'' {\em Journal of Statistical Physics}, vol.~39, no.~1-2,
  pp.~73--131, 1985.

\bibitem{tsitsiklis}
J.~N. Tsitsiklis, ``Markov chains with rare transitions and simulated
  annealing,'' {\em Mathematics of Operations Research}, vol.~14, no.~1,
  pp.~70--90, 1989.

\bibitem{holley}
R.~Holley and D.~Stroock, ``Simulated annealing via sobolev inequalities,''
  {\em Communications in Mathematical Physics}, vol.~115, no.~4, pp.~553--569,
  1988.

\bibitem{gelfand}
S.~B. Gelfand and S.~K. Mitter, ``Simulated annealing with noisy or imprecise
  energy measurements,'' {\em Journal of Optimization Theory and Applications},
  vol.~62, no.~1, pp.~49--62, 1989.

\bibitem{gutjahr}
W.~J. Gutjahr and G.~C. Pflug, ``Simulated annealing for noisy cost
  functions,'' {\em Journal of global optimization}, vol.~8, no.~1, pp.~1--13,
  1996.

\bibitem{bout}
C.~Bouttier and I.~Gavra, ``Convergence rate of a simulated annealing algorithm
  with noisy observations,'' {\em The Journal of Machine Learning Research},
  vol.~20, no.~1, pp.~127--171, 2019.

\bibitem{branke}
J.~Branke, S.~Meisel, and C.~Schmidt, ``Simulated annealing in the presence of
  noise,'' {\em Journal of Heuristics}, vol.~14, no.~6, pp.~627--654, 2008.

\bibitem{lug}
N.~Cesa-Bianchi and G.~Lugosi, {\em Prediction, learning, and games}.
\newblock Cambridge university press, 2006.

\bibitem{john1}
D.~S. Johnson, C.~R. Aragon, L.~A. McGeoch, and C.~Schevon, ``Optimization by
  simulated annealing: An experimental evaluation; part i, graph
  partitioning,'' {\em Operations research}, vol.~37, no.~6, pp.~865--892,
  1989.

\bibitem{john2}
D.~S. Johnson, C.~R. Aragon, L.~A. McGeoch, and C.~Schevon, ``Optimization by
  simulated annealing: an experimental evaluation; part ii, graph coloring and
  number partitioning,'' {\em Operations research}, vol.~39, no.~3,
  pp.~378--406, 1991.

\bibitem{koul}
C.~Koulamas, S.~Antony, and R.~Jaen, ``A survey of simulated annealing
  applications to operations research problems,'' {\em Omega}, vol.~22, no.~1,
  pp.~41--56, 1994.

\bibitem{avr}
K.~Avrachenkov, V.~S. Borkar, S.~Moharir, and S.~M. Shah, ``Dynamic social
  learning under graph constraints,'' 2021.

\bibitem{pem}
R.~Pemantle, ``Vertex-reinforced random walk,'' {\em Probability Theory and
  Related Fields}, vol.~92, no.~1, pp.~117--136, 1992.

\bibitem{benaim}
M.~Bena{\"\i}m {\em et~al.}, ``Vertex-reinforced random walks and a conjecture
  of pemantle,'' {\em The Annals of Probability}, vol.~25, no.~1, pp.~361--392,
  1997.

\bibitem{borkar}
V.~S. Borkar, {\em Stochastic approximation: a dynamical systems viewpoint},
  vol.~48.
\newblock Springer, 2009.

\bibitem{hof}
J.~Hofbauer, K.~Sigmund, {\em et~al.}, {\em Evolutionary games and population
  dynamics}.
\newblock Cambridge university press, 1998.

\bibitem{sand}
W.~H. Sandholm, {\em Population games and evolutionary dynamics}.
\newblock MIT press, 2010.

\bibitem{issac}
D.~L. Isaacson and R.~W. Madsen, {\em Markov chains theory and applications}.
\newblock 1976.

\bibitem{seneta}
E.~Seneta, {\em Non-negative matrices and Markov chains}.
\newblock Springer Science \& Business Media, 2006.

\bibitem{chiang}
T.-S. Chiang and Y.~Chow, ``On eigenvalues and annealing rates,'' {\em
  Mathematics of operations research}, vol.~13, no.~3, pp.~508--511, 1988.

\bibitem{vent}
A.~D. Venttsel, ``On the asymptotics of eigenvalues of matrices with elements
  of order e$\{$-Vij/(2$\varepsilon$\^{}2)$\backslash$$\}$,'' in {\em Doklady
  Akademii Nauk}, vol.~202, pp.~263--265, Russian Academy of Sciences, 1972.

\bibitem{tsitsurvey}
J.~N. Tsitsiklis, ``A survey of large time asymptotics of simulated annealing
  algorithms,'' in {\em Stochastic Differential Systems, Stochastic Control
  Theory and Applications}, pp.~583--599, Springer, 1988.

\bibitem{strockbook}
D.~W. Stroock, {\em An introduction to Markov processes}, vol.~230.
\newblock Springer Science \& Business Media, 2013.

\bibitem{krup}
V.~Kuleshov and D.~Precup, ``Algorithms for multi-armed bandit problems,'' {\em
  arXiv preprint arXiv:1402.6028}, 2014.

\end{thebibliography}

\end{document}